\documentclass[11pt,letterpaper]{article}
\usepackage[T1]{fontenc}
\usepackage[utf8]{inputenc}
\usepackage[letterpaper,margin=1in]{geometry}
\usepackage{lmodern}
\usepackage{microtype}
\usepackage{amsmath,amssymb,amsthm,mathtools}
\usepackage{empheq}
\usepackage[inline]{enumitem}
\usepackage[mathlines]{lineno}
\usepackage{tikz}
\usetikzlibrary{arrows.meta}
\usetikzlibrary{math}
\usepackage{pgfplots}

\newlist{mylist}{enumerate}{1}
\setlist[mylist,1]{label = \textup{(\Roman*)},left = 0pt}

\usepackage{authblk}

\usepackage[colorlinks=true,linkcolor=blue,citecolor=blue,urlcolor=blue]{hyperref}
\usepackage{cleveref}

\newtheorem{theorem}{Theorem}

\newtheorem{lemma}[theorem]{Lemma}

\newtheorem{definition}[theorem]{Definition}
\newtheorem{assumption}{Assumption}

\newtheorem{claim}[theorem]{Claim}

\newcommand{\bbN}{\mathbb{N}}
\newcommand{\bbR}{\mathbb{R}}
\newcommand{\bbZ}{\mathbb{Z}}

\newcommand{\cB}{\mathcal{B}}

\newcommand{\cK}{\mathcal{K}}
\newcommand{\cI}{\mathcal{I}}
\newcommand{\cO}{\mathcal{O}}
\newcommand{\cS}{\mathcal{S}}
\newcommand{\cW}{\mathcal{W}}
\newcommand{\cX}{\mathcal{X}}

\newcommand{\numofpairs}{d_1}

\renewcommand{\epsilon}{\varepsilon}

\newcommand{\argmin}{\operatorname*{argmin}}
\newcommand{\dist}{\operatorname{dist}}
\newcommand{\ip}[2]{\left\langle #1,#2\right\rangle}

\DeclarePairedDelimiter{\abs}{\lvert}{\rvert}

\DeclarePairedDelimiter{\norm}{\lVert}{\rVert}

\DeclareMathOperator{\diag}{diag}

\DeclareMathOperator{\argmax}{argmax}
\newcommand{\myargmax}[1]{\underset{#1}{\argmax}}

\newcounter{relmark}%
\renewcommand{\therelmark}{\alph{relmark}}
\NewProperty{relmark}{now}{??}{\therelmark}
\newcommand{\markrel}[2]{%
  \stepcounter{relmark}%
  \RecordProperties{#1}{relmark}%
  \overset{{(\textup{\therelmark})}}{#2}%
}
\newcommand{\relref}[1]{%
  (\textup{\RefProperty{#1}{relmark}})%
}

\title{A Minimax Analysis of the Last Iterate of the SubGradient Method in Finite Dimension}

\author[1]{Guglielmo Beretta}
\author[2]{Tommaso Cesari}
\author[3]{Roberto Colomboni}
\author[1]{Andrea Paudice}

\affil[1]{Department of Computer Science, Aarhus University, Aarhus, Denmark}
\affil[2]{School of Electrical Engineering and Computer Science, University of Ottawa, Ottawa, Canada}
\affil[3]{School of Mathematics, University of Bristol, Bristol, United Kingdom}

\affil[ ]{
\texttt{gberetta@cs.au.dk},
\texttt{tcesari@uottawa.ca},
\texttt{roberto.colomboni@bristol.ac.uk},
\texttt{apaudice@cs.au.dk}
}

\date{}

\begin{document}

\maketitle
\begin{abstract}
\noindent 
We study the last iterate of the projected \emph{subGradient Method} with a constant stepsize for convex $L$-Lipschitz objectives in $\mathbb{R}^d$, minimized over closed convex sets. Assuming that the initial iterate $x_1$ has distance at most $R$
from a minimizer, we establish matching upper and lower bounds on the \emph{minimax optimization error} at $x_n$, where the infimum is taken over constant stepsizes. Specifically, this error is of order $L\cdot R\cdot\sqrt{\min\{d,\log n\}/n}$. Notably, when the dimension is fixed, the minimax error decays as $1/\sqrt{n}$, without the extra logarithmic factor intrinsic to dimension-free analyses. Our analysis exploits a combinatorial argument based on linear algebra and novel geometric constructions.
\end{abstract}
\section{Introduction}\label{sec:introduction}
\noindent
The \emph{subGradient Method} (sGM), originating in the work of Shor \cite{Shor1962}, is one of the simplest first-order methods for nonsmooth convex optimization. Its appeal comes from its minimal structure: each update uses a subgradient of the objective and a stepsize, together with a projection in the constrained setting. This simplicity has made subgradient-type methods a standard workhorse in large-scale optimization, machine learning, and statistics, for instance in empirical risk minimization problems where nonsmooth regularizers or losses naturally arise. For given $L,R>0$, we consider a convex objective $f:\bbR^d\to\bbR$ that is $L$-Lipschitz in the Euclidean norm, minimized over a closed convex set $\cX\subseteq\bbR^d$ that need not be bounded. Given an initial point $x_1\in\cX$ satisfying $\|x_1-x_\star\|\leq R$ for some minimizer $x_\star$ of $f$ over $\cX$, and a prescribed horizon $n$, the projected sGM generates the iterates according to
\begin{equation}\label{eq:sgm_intro}
\forall 1 \leq k < n, \quad
x_{k+1} = \Pi_{\cX}(x_k - \eta \cdot g_k),
\end{equation}
where $\Pi_{\cX}$ denotes the Euclidean projection onto $\cX$, and $g_k$ is a subgradient of $f$ at $x_k$. We focus on the \emph{constant-stepsize} setting: the coefficient $\eta>0$ multiplying $g_k$ remains fixed throughout the run, although it may depend on $n,d,L,R$. We measure the optimization error at $x\in\cX$ by $\Delta(x) \coloneqq f(x)-f(x_\star)$.

\noindent
Unlike gradient descent for smooth convex objectives with a suitable stepsize, sGM is not a descent method: the objective value may increase from one iteration to the next \cite[Section~8.1.1]{Beck2017}. Classical non-asymptotic guarantees therefore concern either the best-iterate error $\min_{1\leq k\leq n}\Delta(x_k)$ or the averaged-iterate error $\Delta(n^{-1}\sum_{k=1}^n x_k)$. In particular, a constant stepsize of order $R/(L \cdot \sqrt{n})$ bounds both quantities by a universal constant times $(L \cdot R)/\sqrt{n}$, see \cite{Shor1968,Polyak1969}. This rate is independent of the dimension and is optimal in the dimension-free first-order black-box model \cite{Nemirovski1983}. However, these guarantees do not directly control the last-iterate error $\Delta(x_n)$.

\noindent
The convergence of the last iterate has consequently received increasing attention \cite{Shamir2013,Harvey2019,Jain2021,Zamani2025}. Although $x_n$ is the natural output of the algorithm, its worst-case behavior is substantially more delicate than that of the best or averaged iterate. For the standard constant stepsize of order $R/(L\sqrt{n})$, the sharp dimension-free worst-case error is of order $L \cdot R \cdot (\log n)/\sqrt{n}$, see \cite{Shamir2013,Lin2018,Zamani2025}. Even when the constant stepsize is optimized for the last iterate, the sharp dimension-free rate remains of order $L \cdot R \cdot \sqrt{(\log n)/n}$, see \cite{Zamani2025}. 
The lower-bound constructions establishing the necessity of these extra logarithmic factors for constant stepsizes require a dimension that grows with $n$, and therefore do not settle this question in fixed dimension. In \cite{Koren2020}, Koren and Segal observed that, in dimension one, the standard constant stepsize yields a last-iterate error bounded by a constant times $(L \cdot R)/\sqrt{n}$. In their Open Problem 2, they asked for the corresponding convergence rate in fixed dimension $d>1$. \cite{Liu2021} subsequently established a lower bound of order $(L \cdot R \cdot \log d)/\sqrt{n}$ for the standard constant stepsize in the regime $2\leq d<n$, already in the deterministic setting. While this shows that some dependence on the dimension is unavoidable, it leaves open both whether extra logarithmic factors in the horizon can be avoided in fixed dimension and what the sharp dependence on $d$ is.

\paragraph{Contributions.}
In this work, we provide, to the best of our knowledge, the first sharp dimension-dependent characterization of the worst-case last-iterate error of sGM with a constant stepsize. We first establish a dimension-dependent upper bound on $\Delta(x_n)$ that holds for every constant stepsize. With a suitable choice of stepsize, this bound recovers the classical $1/\sqrt{n}$ dependence on the horizon in every fixed dimension, without extra logarithmic factors (\Cref{thm:upper-bound}). We complement this upper bound with lower bounds that remain valid after optimizing the constant stepsize (\Cref{thm:global_minimax_lower_bound}). For fixed $d,L,n,R$, let $\Delta_{\star}(d,L,n,R)$ denote the infimum, over constant stepsizes $\eta>0$, of the worst-case optimization error $\Delta(x_n)$ over the problem class described above. Together with the sharp dimension-free upper bound in \cite{Zamani2025}, our results yield, for every $d \geq 1$, every $n \ge 2$ and every $L$, $R > 0$,
\begin{equation}\label{eq:main_contrinution}
\Delta_{\star}(d,L,n,R)
\asymp
L\cdot R\cdot
\sqrt{\frac{\min\{d,\log n\}}{n}},
\end{equation}
where $\asymp$ denotes matching upper and lower bounds up to positive multiplicative constants independent of $d,L,n,R$. The $\sqrt{d}$ dependence is unavoidable in the regime $d \leq \log n$, beyond which the error saturates at the sharp dimension-free rate. In particular, our results resolve Open Problem~2 of Koren and Segal~\cite{Koren2020} on last-iterate convergence in fixed dimension.

\section{Related work}
\paragraph{Time-varying stepsizes.}
Suitably designed time-varying stepsizes can remove the extra logarithmic factors from dimension-free last-iterate guarantees. \cite{Jain2021} develop a stepsize policy that transfers guarantees for averaged iterates to the last iterate. In the deterministic setting, their stagewise decreasing schedule yields an error of order at most $(L\cdot D)/\sqrt{n}$, where $D$ is the diameter of the feasible set. \cite{Zamani2025} obtain the optimal bound $(L\cdot R)/\sqrt{n}$ using a linearly decreasing stepsize, without requiring a bounded feasible set. Both constructions depend explicitly on the horizon $n$. A stepsize policy is called \emph{anytime} if it does not depend on a prescribed horizon. On bounded feasible sets, the standard choice $\eta_k=1/\sqrt{k}$ yields error bounds of order $(\log k)/\sqrt{k}$, even in stochastic settings \cite{Shamir2013,Harvey2019}. \cite{Kornowski2026} establish an anytime upper bound of order $L\cdot D\sqrt{(\log k)/k}$ in the deterministic setting, and prove that no predetermined stepsize sequence can achieve a better rate. In fixed dimension, a different picture emerges. Building on the geometric argument introduced in a preliminary version of the present work, \cite{Dewasurendra2026} obtains an anytime upper bound of order $L\cdot D\cdot\sqrt{d/k}$ on feasible sets of diameter at most $D<\infty$, using $\eta_k$ of order $D/(L\cdot\sqrt{d\cdot k})$.

\paragraph{Stochastic subGradient Method.}
The stochastic variant of sGM replaces exact subgradients with conditionally unbiased estimates \cite{Ermoliev1967}. For convex Lipschitz objectives on bounded feasible sets, early work by Shamir and Zhang \cite{Shamir2013} established expected error bounds of order $(\log n)/\sqrt{n}$ under a standard second-moment assumption. Subsequently, \cite{Harvey2019} obtained high-probability bounds with the same dependence on $n$ under almost surely bounded noise. \cite{Eldowa2024} extended such high-probability guarantees to stochastic mirror descent \cite{Nemirovski1983}, allowing for either sub-Weibull noise or noise with uniformly bounded conditional moments of order $p>4$. Time-varying stepsize policies depending on the time horizon can remove the extra logarithmic factor even in some stochastic settings \cite{Jain2021,Liu2024}. For constant stepsizes, however, finite dimension alone does not eliminate the extra logarithmic factor. Indeed, \cite{Beretta2026} establish a lower bound with a $(\log n)/\sqrt{n}$ dependence on the horizon, even in dimension one, for the standard constant stepsize under uniformly bounded conditional noise variance. This contrasts with our deterministic characterization, where the extra logarithmic factor in the horizon disappears in every fixed dimension.

\paragraph{Smooth convex objectives.}
For unconstrained convex objectives with a $\beta$-Lipschitz gradient, gradient descent with a constant stepsize $\eta=1/\beta$ achieves a dimension-free anytime last-iterate upper bound of order $(\beta\cdot R^2)/n$. This rate is sharp even in dimension one \cite{Drori2014}. Recent work shows that it can be improved solely by varying the stepsizes \cite{Altschuler2025,Zhang2025}. In particular, Altschuler and Parrilo introduce the Silver Stepsize Schedule \cite{Altschuler2025}, which yields a last-iterate upper bound of order $(\beta\cdot R^2)/n^{1.271}$. Although the schedule is independent of the horizon, this guarantee holds at the endpoints of its recursive blocks, rather than at every iteration. \cite{Zhang2025} establish a dimension-free anytime bound of order $(\beta\cdot R^2)/n^{1.119}$ using a predetermined horizon-independent stepsize sequence. These results highlight the role of stepsize design in determining the performance of classical first-order methods.

\section{Problem Setting and Results}\label{sec:setting}
\subsection{Assumptions}\label{ssec:setting}

\paragraph{Notation.}
Throughout the paper, $d$, $k$, and $n$ denote positive integers. For $x,y\in\bbR^d$, we write $\norm{x}$ for the Euclidean norm of $x$ and $\ip{x}{y}$ for the Euclidean inner product of $x$ and $y$. For a nonempty closed convex set $\cS\subseteq\bbR^d$, we write $\Pi_{\cS}$ for the Euclidean projection onto $\cS$ and set $\dist(x,\cS)\coloneqq\inf_{z\in\cS}\norm{x-z}=\norm{x-\Pi_{\cS}(x)}$. Finally, we use the notation $a \lesssim b$ (resp. $a \gtrsim b$) with the meaning that $a \leq C \cdot b$ (resp. $a \geq c \cdot b$) for some universal numerical constant $C > 0$ (resp. $c > 0$). Furthermore, we write $a \asymp b$ if and only if $b \lesssim a $ and $ a \lesssim b$.

\medskip
\noindent
We consider the optimization problem
\begin{equation}\label{eq:problem}
\underset{x\in\cX}{\mathrm{minimize}}\; f(x),
\tag{P}
\end{equation}
where $f$ is convex and the optimizer has access to a \emph{first-order oracle} $x\mapsto g(x)$ that, when queried at $x\in\cX$, returns a subgradient of $f$ at $x$.
We make the following standard assumptions on \eqref{eq:problem}:
\begin{assumption}[Convexity and Lipschitz continuity]
\label{ass:lipschitz}
The objective $f:\bbR^d\to\bbR$ is convex and $L$-Lipschitz with respect to the Euclidean norm, for a given constant $L>0$.
\end{assumption}
\noindent
The constant $L$ is part of the problem data and need not be the smallest Lipschitz constant of $f$. We recall the definition of a subgradient.
\begin{definition}[Subgradient]
Let $x\in\bbR^d$. A vector $g\in\bbR^d$ is a \emph{subgradient} of $f$ at $x$ if
\[
\forall y\in\bbR^d,\qquad
f(y)\geq f(x)+\ip{g}{y-x}.
\]
The \emph{subdifferential} of $f$ at $x$, denoted by $\partial f(x)$, is the set of all subgradients of $f$ at $x$.
\end{definition}
\noindent
By Assumption~\ref{ass:lipschitz}, the subdifferential $\partial f(x)$ is nonempty for every $x\in\bbR^d$, and every $g\in\partial f(x)$ satisfies $\norm{g}\leq L$. We formalize the oracle model as follows.
\begin{definition}[First-Order Oracle]
Let $h:\bbR^d\to\bbR$ be convex. A \emph{first-order oracle} for $h$ is a map $g:\bbR^d\to\bbR^d$ satisfying $g(x)\in\partial h(x)$ for every $x\in\bbR^d$. Since the subdifferential need not be a singleton, such an oracle need not be unique. We denote the set of all first-order oracles for $h$ by
\[
\cO_h\coloneqq
\{g:\bbR^d\to\bbR^d:
g \text{ is a first-order oracle for }h\}.
\]
\end{definition}
\noindent
For simplicity, we write $g_k=g(x_k)$ for the oracle response at the iterate $x_k$ generated by \eqref{eq:sgm_intro}.
\begin{assumption}[Feasible Set and Minimizers]
\label{ass:minimizers}
The feasible set $\cX\subseteq\bbR^d$ is nonempty, closed, and convex, and the set of minimizers $\cX_\star= \cX_\star(f, \cX)\coloneqq\arg\min_{x\in\cX}f(x)$ is nonempty.
\end{assumption}
\noindent
We define $f_\star = f_\star(f, \cX)\coloneqq\min_{x\in\cX}f(x)$, which is finite by Assumption~\ref{ass:minimizers}. As in the introduction, the optimization error at $x\in\cX$ is $\Delta(x) = \Delta(x;f,\cX) \coloneq f(x)-f_\star$.
\paragraph{Admissible Problem Instances and Minimax Optimization Error.}
Fix $d\geq1$ and $L,R>0$. We denote by $\cI(d,L)$ the class of admissible problem instances, consisting of all triples $(f,g_f,\cX)$ such that $f$ satisfies Assumption~\ref{ass:lipschitz} with the prescribed Lipschitz bound $L$, $g_f\in\cO_f$, and $\cX$ satisfies Assumption~\ref{ass:minimizers}. For each instance, we restrict the initial point to the set $B_\star(R)=B_\star(R; f, \cX)\coloneqq\{x\in\cX:\dist(x,\cX_\star(f,\cX))\leq R\}$. Given $\eta>0$ and $x_1\in B_\star(R)$, let $x_n^{\rm sGM}(\eta,x_1;g_f,\cX)$ denote the iterate obtained after $n-1$ updates of \eqref{eq:sgm_intro}, starting from $x_1$ and using the constant stepsize $\eta$ and the oracle $g_f$. In this work, we focus on the \emph{minimax optimization error} of the last iterate of sGM defined as
\begin{equation}\label{eq:minimax_error}
\Delta_\star(d, L, n, R) \coloneqq \inf_{\eta>0} \sup_{\substack{ (f, g_f, \cX)\in\cI(d, L)\\ x_1\in B_\star(R; f, \cX) }} \Delta\bigl(x_n^{\rm sGM}(\eta, x_1; g_f, \cX); f, \cX \bigr) \;.
\end{equation}
Note that the stepsize $\eta$ may depend on $(d,L,n,R)$, but must be chosen uniformly over all admissible instances and initial points.

\subsection{Results}
We begin with our main upper bound, which explicitly captures the dependence of the last-iterate error of sGM on the dimension.
\begin{theorem}\label{thm:upper-bound}
Let $(f,g,\cX) \in \cI(d,L)$, let $x_1\in \cX$, let $\eta>0$ and define $x_k \coloneq x_k^{\rm sGM}(\eta, x_1;g,\cX)$ for every $k \ge 1$. Then
\begin{equation}\label{eq:upper-bound}
f(x_n) - f_{\star}
\le
\bigg( d + \frac{1}{2}\bigg) \cdot L^2 \cdot \eta
+
\frac 12 \cdot \frac{\dist(x_1, \cX_{\star})^{2}}{n} \cdot \frac 1\eta
\end{equation}
\end{theorem}
\noindent We make the following observations. First, eq.~\eqref{eq:upper-bound} implies that the standard choice $\eta = c/\sqrt{n}$ for any fixed constant $c > 0$, yields the classical $1/\sqrt{n}$ dependence on the horizon in every fixed dimension, without extra logarithmic factors. Second, if the upper bound $R$ on $\dist(x_1, \cX_{\star})$ is known, then replacing $\dist(x_1,\cX_\star)^2$ with $R^2$ in \eqref{eq:upper-bound} and minimizing the resulting upper bound over $\eta$ gives the choice $\eta=R/(L \cdot \sqrt{(2 \cdot d+1) \cdot n})$, for which $\Delta(x_n)\leq L \cdot R \cdot \sqrt{(2 \cdot d+1)/n}$. This stepsize is $1/\sqrt{2 \cdot d + 1}$ times the standard choice and improves the dimension dependence of the upper bound from linear to square root. Combining this guarantee with the optimized dimension-free bound of Zamani and Glineur \cite{Zamani2025} yields
\begin{equation}\label{eq:ub_minmax}
\Delta_\star(d,L,n,R)
\leq
 \sqrt{3} \cdot L\cdot R\cdot
\sqrt{\frac{\min\{d,\log n\}}{n}}\;.
\end{equation}
which holds for every $d \geq 1, L > 0, n \geq 2$, and $R > 0$.
\noindent
Finally, the key ingredient in the proof of \Cref{thm:upper-bound} is a dimension-dependent bound on the gap between the last and best iterates. Let $k_\star$ be the largest index attaining $\min_{1\leq k\leq n}f(x_k)$. We show that
\begin{equation*}
f(x_n) - f(x_{k_\star}) \le d \cdot \eta \cdot L^2 \,.
\end{equation*}
For fixed $\eta$, this bound is independent of both the horizon and the initial distance to the minimizer set. Combined with the classical best-iterate guarantee and a suitable choice of $\eta$, it yields the $1/\sqrt{n}$ dependence on the horizon without extra logarithmic factors in fixed dimension.
Classical dimension-free analyses compare the last objective value with suffix averages, and, for standard stepsizes, produce logarithmic factors in $n$ through harmonic
summations~\cite{Shamir2013,Lin2018,Harvey2019}. \cite{Jain2021} obtain a sharper comparison through a suitable time-varying stepsize schedule. Our geometric argument instead keeps the stepsize constant and exploits the finite-dimensional structure of $\bbR^d$ through the use of a combinatorial construction based on linear algebra  (see \Cref{sec:upper-bound-proof}).

\noindent We next establish a complementary minimax lower bound.
\begin{theorem}\label{thm:global_minimax_lower_bound}
For every $d \geq 1$, every $n \ge 2$ and every $L$, $R > 0$,
\begin{equation}\label{eq:global_minimax_lower_bound}
\Delta_\star(d,L,n,R)
\geq
\frac{1}{8\sqrt{2 \log 10}} \cdot L \cdot R \cdot \sqrt{\frac{\min\{d,\log n\}}{n}}.
\end{equation}
\end{theorem}
\noindent We start noting that in the regime $n \geq \exp(d)$, the bounds in \eqref{eq:global_minimax_lower_bound} and \eqref{eq:ub_minmax}, imply that $\Delta_{\star} \asymp L \cdot R \cdot \sqrt{d/n}$. This settles the dependence on the horizon in Open Problem~2 posed by Koren and Segal at COLT 2020 \cite{Koren2020} and also refutes their conjectured logarithmic dependence on the dimension, already in the deterministic setting. We also note that Theorem~\ref{thm:global_minimax_lower_bound}, together with \eqref{eq:ub_minmax}, establishes the sharp minimax rate across all dimensions and horizons.  Finally, the proof of \Cref{thm:global_minimax_lower_bound} differs substantially from the classical high-dimensional constructions underlying dimension-free oracle lower bounds \cite{Nemirovski1983} and sharp last-iterate lower bounds \cite{Harvey2019,Harvey2024,Zamani2025}. These constructions require an ambient dimension at least proportional to the iteration horizon and therefore do not capture the behavior of the method over arbitrarily long runs in fixed dimension. In contrast, for each fixed $d\geq3$ and every sufficiently large horizon, our construction produces hard instances in $\bbR^d$ that establish the sharp $\sqrt{d}$ dependence even after optimizing the constant stepsize. The central challenge is to obtain this dimension-dependent error without introducing a new independent direction at each iteration.
\section{Upper Bound}\label{sec:upper-bound-proof}
Our proof combines the classical best-iterate guarantee with a geometric argument that exploits the finite dimension of the ambient space. We begin by recalling the former; see, for example, \cite[Lemma~8.24]{Beck2017}.
\begin{theorem}\label{thm:best-iterate}
Let $(f,g,\cX) \in \cI(d,L)$, let $x_1\in \cX$ and let $\eta>0$. Then
\begin{equation}\label{eq:std-best} 
\min_{ 1 \le k  \le n} \Delta(x_k) 
\le
\frac 12 \cdot \left( L^2 \cdot \eta 
+
\frac{\dist(x_1, \cX_{\star})^{2}}{n} \cdot \frac 1\eta \right) \,.
\end{equation}
\end{theorem}
\noindent
The dimension-dependent part of the proof relies on the following classical linear-algebra lemma; see, for example, \cite[Ch.~V, Sec.~8, Lemma~4]{Serre1987}.
\begin{lemma}\label{lem:rank}
Let $v_1$, $\dots$, $v_M \in \bbR^d$ such that $\ip{v_i}{v_j} \le 0$ for every $1 \le i < j \le M$. Suppose there exists $u\in\bbR^d$ such that
$\ip{v_\ell}{u}>0$ for every $1\le \ell\le M$. Then $v_1$, $\dots$, $v_M$ are linearly independent.
\end{lemma}
\noindent 
In other words, vectors with pairwise nonpositive inner products are linearly independent if they lie strictly on the same side of a hyperplane through the origin. In particular, such a family contains at most $d$ vectors. This cardinality bound is the key to the combinatorial argument developed below. For completeness, we provide a proof of Lemma~\ref{lem:rank} in Appendix~\ref{sec:rank}.

\noindent
The geometric construction also uses the fact that sufficiently large objective gaps force the iterates to move closer to a comparison point. The following lemma gives sufficient conditions for this strict decrease in distance over one or two consecutive projected subgradient updates. Both statements follow directly from the standard one-step inequality for the subgradient method.
\begin{lemma}\label{lem:1-2step}
Let $u\in \cX$ and let $y_1$, $y_2$, $y_3 \in \cX$. For every $k \in \{ 1, 2\}$, assume $y_{k+1} = \Pi_{\cX} (y_k - \alpha_k \cdot g_k')$ for some scalars $\alpha_k > 0$ and some $g_k' \in \partial f(y_k)$.
Then:
\begin{mylist}
\item \label{it:1step} $f(y_1) > f(u) +  \frac{1}{2} \cdot \alpha_1 \cdot L^2$ implies  $\norm{y_2 - u} < \norm{y_{1} - u}$;
\item \label{it:2step} $\sum_{k = 1}^2 \alpha_k^2 \cdot L^2 - 2 \cdot \sum_{k = 1}^2 \alpha_k \cdot \big(f(y_k) - f(u) \big) < 0$ implies $\norm{y_3 - u} < \norm{y_{1} - u}$.
\end{mylist}
\end{lemma}
\subsection{Proof Sketch of Theorem~\ref{thm:upper-bound}}
To illustrate the main geometrical idea on which Theorem~\ref{thm:upper-bound} is based, we sketch a proof describing a slightly weaker bound, while we defer to Section~\ref{sec:upper-thm-proof} how the same argument can be refined to prove in full Theorem~\ref{thm:upper-bound}.
\begin{proof}[Proof sketch]
First, we select the \emph{last best iterate}, i.e., the iterate $x_{k_\star}$ among $x_1$, $x_2$, $\dots$, $x_n$ where $k_{\star}$ is chosen as the largest integer for which $f(x_{k_\star}) = \min_{ 1 \le k  \le n} f(x_k)$. Decomposing the optimization error, we can bound $f(x_n) - f_{\star}$ by the sum of two terms: one is $f(x_n) - f(x_{k_\star})$, on which our analysis is focused; the other is $f(x_{k_\star}) - f_\star$, which can be bounded by Theorem~\ref{thm:best-iterate}, so that \eqref{eq:upper-bound} follows if we show that 
\begin{equation}\label{eq:upper-trail-bis}
f(x_n) - f(x_{k_\star})
\le
d \cdot \eta \cdot L^2 \,.
\end{equation}
After treating the case $d=1$ thanks to an argument explained in \cite{Koren2020}, we can assume $d \ge 2$. Let $q$ be the minimum integer such that
\[
f(x_n) - f(x_{k_\star})
\le
q \cdot \eta \cdot L^2 \,,
\]
so that the thesis holds if we show that $q \le d$. However, in this proof sketch, we will only describe how the coarser bound $q \le d+1$ can be proved.
The idea is to select, among the iterates, $q-1$ distinct points $x_{\tau(0)}$, $x_{\tau(1)}$, $\dots$, $x_{\tau(q-2)}$ so that, defining for every $0 \le i < j \le q-2$ the midpoints
\[
u_{i,j} \coloneq \frac{x_{\tau(i)} + x_{\tau(j)}}{2}\,,
\]
the following inequalities hold:
\begin{subequations}\label{eq:conditions-on-xn-bis}
        \begin{empheq}[left={\empheqlbrace}]{alignat = 2}
        \norm{x_{n} - u_{i,j} } &< \frac{1}{2}\cdot\norm{x_{\tau(i)} - x_{\tau(j)}}  & \qquad&\text{for every } \quad 0 \le i< j \le q - 2\,,
        \label{eq:bound-on-xn-sys-bis}\\
        f(x_{n}) &> f(x_{\tau(\ell)}) & &\text{for every } \quad 0 \le \ell \le q - 2\,.
        \label{eq:highfxn-bis}
        \end{empheq}
\end{subequations}
We can give \eqref{eq:bound-on-xn-sys-bis} a geometrical interpretation. Consider in $\bbR^d$ the closed ball $\cK_{i,j}$ with center $u_{i,j}$ and radius 
$(1/2) \cdot \norm{x_{\tau(i)} - x_{\tau(j)}}$. Then \eqref{eq:bound-on-xn-sys-bis} says that $x_n$ lies in the interior of $ \cK_{i,j}$ for $0 \le i< j \le q - 2$, and since $x_{\tau(i)}$ and $x_{\tau(j)}$ are antipodal points on $\cK_{i,j}$ by construction, this gives
\[
\ip{x_n - x_{\tau(i)} }{x_n - x_{\tau(j)}}< 0\,,
\]
i.e., the vectors $w_\ell\coloneq x_{n} - x_{\tau(\ell)}$, with $0\le \ell \le q-2$, are distinct and form obtuse angles pairwise.\\
There is another geometric property regarding the vectors $w_0$, $w_1$, $\dots$, $w_{q-2}$ that is obtained by \eqref{eq:highfxn-bis} and the convexity of $f$.
Choose $g_n \in \partial f (x_n)$. The definition of subgradient gives
\begin{equation}\label{eq:subgr-hn-bis}
\ip{g_n}{x_n - y} \ge f(x_n) - f(y) \qquad \text{for every } y \in \bbR^d \,.
\end{equation}
Combining \eqref{eq:subgr-hn-bis} with \eqref{eq:highfxn-bis} we obtain $\ip{g_n}{w_\ell}>0$ for every $ 0 \le \ell \le q - 2$.
In words, $g_n$ defines an open half-space containing all the vectors $w_0, w_1, \dots, w_{q-2}$. 
By Lemma~\ref{lem:rank}, we deduce the linear independence of $w_0, w_1, \dots, w_{q-2}$. Therefore their number, i.e., $q - 1$, cannot exceed the dimension $d$, and the claimed bound $q \le d + 1$ follows.
What is left is proving how to choose the iterates  $x_{\tau(0)}$, $x_{\tau(1)}$, $\dots$, $x_{\tau(q-2)}$ so that \eqref{eq:conditions-on-xn-bis} hold.
To this end, we partition $\bbR$ into bins of the form $\cB(i) \coloneq ((i-1) \eta L^2 + f(x_{k_\star}),  i \eta L^2 + f(x_{k_\star})]$, with $i \in \bbZ$, and we define, for every $0 \le i \le q$,
\[
\tau(i) \coloneq \max\{ k \in \bbN \;|\; k_{\star} \le  k \le n, \; f(x_k) \in \cB(i) \}\,,
\]
i.e., $\tau(i)$ is the last time $f(x_k)$ visits $\cB(i)$ for $k_{\star} \le  k \le n$, see Figure~\ref{fig:thresholds}.
\begin{figure}
    \centering
\begin{tikzpicture}[scale=0.7]
\pgfplotsset{compat=1.18}
\usetikzlibrary{patterns}
\begin{axis}[
    width=15cm,
    height=10cm,
    xlabel={$k$},
    xlabel style={
        at={(axis description cs:1,0)},
        anchor=north east,
        rotate=0
    },    
    ylabel={$f(x_k)$},
    ylabel style={
        at={(axis description cs:0,1)},
        anchor=south west,
        rotate=-90
    },    
    xmin=-0.5, xmax=8,
    ymin=-1, ymax=7,
    xtick={0,2.5,4.5,6.0},
    xticklabels={
        $\tau(0)$,
        $\tau(1)$,
        $\tau(2)$,
        $\tau(3)$,
    },
    ytick={0,2,4,6},
    yticklabels={
        $f(x_{k_{\star}})$,
        $\eta L^2 + f(x_{k_{\star}})$,
        $2\eta L^2 + f(x_{k_{\star}})$,
        $3\eta L^2 + f(x_{k_{\star}})$,
    },
    axis lines=left,
    tick align=outside,
    clip=false,
]

\addplot[
    only marks,
    mark=*,
    mark size=2pt,
] coordinates {
    (0.0,0.0)%
    (0.5,1.5)
    (1.0,2.7)
    (1.5,2.5)
    (2.0,3.4)
    (2.5,1.6)%
    (3.0,2.8)
    (3.5,3.1)
    (4.0,4.5)
    (4.5,4.0)%
    (5.0,5.0)
    (5.5,5.3)
    (6.0,5.4)%
    (6.5,6.7)
    (7.0,6.4)
};

\addplot[
    domain=-0.5:7.5,
    samples=2,
] {0};

\addplot[
    domain=-0.5:7.5,
    samples=2,
] {2};

\addplot[
    domain=-0.5:7.5,
    samples=2,
] {4};

\addplot[
    domain=-0.5:7.5,
    samples=2,
] {6};

\draw[dashed]
    (axis cs:0,0) -- (axis cs:0,-1);
\draw[dashed]
    (axis cs:2.5,1.6) -- (axis cs:2.5,-1);
\draw
    (axis cs:2.5,2.0) -- (axis cs:2.5,1.6);
\draw[dashed]
    (axis cs:4.5,4.0) -- (axis cs:4.5,-1);
\draw[dashed]
    (axis cs:6.0,5.4) -- (axis cs:6.0,-1);
\draw
    (axis cs:6.0,6.0) -- (axis cs:6.0,5.4);

\fill[pattern=north west lines, pattern color=gray] (0,0) rectangle (2.5,-1);
\fill[pattern=north west lines, pattern color=gray] (2.5,2.0) rectangle (4.5,-1);
\fill[pattern=north west lines, pattern color=gray] (4.5,4.0) rectangle (6.0,-1);
\fill[pattern=north west lines, pattern color=gray] (6.0,6.0) rectangle (7.5,-1);

\draw[blue,thick] (axis cs:4.0,1.1) -- (axis cs:4.0,5.1);
\draw[blue,thick] (axis cs:3.9,1.1) -- (axis cs:4.1,1.1);
\draw[blue,thick] (axis cs:3.9,5.1) -- (axis cs:4.1,5.1);
\draw[blue,thick] (axis cs:3.9,1.1) -- (axis cs:3.9,1.2);
\draw[blue,thick] (axis cs:4.1,1.1) -- (axis cs:4.1,1.2);
\draw[blue,thick] (axis cs:3.9,5.1) -- (axis cs:3.9,5.0);
\draw[blue,thick] (axis cs:4.1,5.1) -- (axis cs:4.1,5.0);

\draw[blue,dashed] (axis cs:4.0,3.1) -- (axis cs:3.5,3.1);

\end{axis}
\end{tikzpicture}
\caption{%
Last visits to objective-value bins along the sGM trajectory. For every $0\leq i\leq q$, $x_{\tau(i)}$ is the last iterate whose objective value belongs to the bin $B(i)$, of width $\eta \cdot L^2$. The shaded regions indicate values excluded by $f(x_k) > i \cdot \eta \cdot  L^2+f(x_{k_\star})$ whenever $\tau(i)<k\leq n$. The blue interval illustrates the bound $|f(x_{k+1})-f(x_k)| \leq \eta \cdot L^2$, which prevents consecutive objective values from skipping a bin.}
\label{fig:thresholds}
\end{figure}
\noindent The proof that $x_{\tau(0)}$, $x_{\tau(1)}$, $\dots$, $x_{\tau(q-2)}$ are well defined relies on the following bound, valid at every step of sGM:
\begin{equation}\label{eq:onestep-etaL2-bis}
\abs{f(x_{k+1}) - f(x_k)} \le \eta \cdot L^2 \,.
\end{equation}
Indeed, by \eqref{eq:onestep-etaL2-bis}, the objective computed along the iterates can jump out of a bin $\cB(i)$ only if it lands on one of the adjacent bins, either $\cB(i-1)$ or $\cB(i+1)$.
When $k$ ranges between $k_{\star}$ and $n$, the first bin visited by $f(x_k)$ is $\cB(0)$ and the last is $\cB(q)$, thus all bins $\cB(i)$ indexed by $i = 0$, $1$, $\dots$, $q$ are visited by $f(x_k)$ at least once. We deduce that
\[
k_{\star} = \tau(0) < \tau(1) < \dots <\tau(q-2) < \tau(q-1) <\tau(q)=n\,,
\]
and, by the bins definition, the $q + 1$ iterates $x_{\tau(0)}$, $x_{\tau(1)}$, $\dots$, $x_{\tau(q)}$ are $q+1$ distinct vectors since
\[
f(x_{\tau(0)}) < f(x_{\tau(1)}) <  \dots <  f(x_{\tau(q-2)}) <  f(x_{\tau(q-1)})  <  f(x_{\tau(q)}) =f(x_n)\,,
\]
and the condition \eqref{eq:highfxn-bis} holds. What is left is understanding why \eqref{eq:bound-on-xn-sys-bis} holds.
The idea here is that, given a reference point $u \in \bbR^d$, sGM generates $x_{k+1}$ closer than $x_k$ to $u$ provided $f(x_k)$ is sufficiently larger than $f(u)$.
We can prove by convexity of $f$ that, given $0 \le i<j \le q-2$, 
\begin{equation}\label{eq:fuij-tauonehalf-bis}
 f(u_{i,j}) \le \bigg( j - \frac{1}{2} \bigg) \cdot \eta \cdot L^2 + f(x_{k_\star})\,, 
\end{equation}
and since all the iterates $x_k$ \emph{following} $x_{\tau(j)}$ must satisfy  $f(x_k) > j \cdot \eta \cdot L^2 + f(x_{k_\star})$, this information combined with \eqref{eq:fuij-tauonehalf-bis} guarantees (see the classical result in Lemma~\ref{lem:1-2step}~\ref{it:1step}) that
\begin{equation}\label{eq:distance-decreases-bis}
(\norm{x_{k} - u_{i,j}})_{k = \tau(j) + 1}^n \qquad \text{is a decreasing sequence,}
\end{equation}
thus the last term of the sequence, i.e., $\norm{x_{n} - u_{i,j}}$, is strictly smaller than the previous terms.
If we could include $k = \tau(j)$ in \eqref{eq:distance-decreases-bis}, then we would already have \eqref{eq:bound-on-xn-sys-bis}, since $\norm{x_{\tau(j)} - u_{i,j}}$ equals the radius of $\cK_{i,j}$. Unfortunately, we cannot guarantee that $f(x_{\tau(j)})$ is sufficiently larger than
$f(u_{i,j})$ to deduce that $\norm{x_{\tau(j)} - u_{i,j}} > \norm{x_{\tau(j) + 1} - u_{i,j}}$; however, Lemma~\ref{lem:1-2step}~\ref{it:2step} can be used to show that $\norm{x_{\tau(j)} - u_{i,j}} > \norm{x_{\tau(j) + 2} - u_{i,j}}$, that is, the distance from $u_{i,j}$ is provably smaller after two steps of sGM, see Figure~\ref{fig:circle}. We then obtain that $\norm{x_{\tau(j)} - u_{i,j}} > \norm{x_{n} - u_{i,j}}$, which gives \eqref{eq:bound-on-xn-sys-bis}.
\begin{figure}
    \centering

\begin{tikzpicture}[
    point/.style={circle, fill=black, inner sep=1.5pt},
    >=Stealth
]

\tikzmath{
    \r1 = 2.3;
    \r2 = 2.1;
    \r3 = 1.8;
    \r4 = 3.3;
    \r5 = 2.85;
    \x1 = \r1*cos(20);
    \y1 = \r1*sin(20);
    \x2 = \r2*cos(50);
    \y2 = \r2*sin(50);
    \x3 = \r3*cos(80);
    \y3 = \r3*sin(80);
    \x4 = \r4*cos(10);
    \y4 = \r4*sin(10);
    \x5 = \r5*cos(25);
    \y5 = \r5*sin(25);
}

\coordinate (A) at (-3,0);
\coordinate (B) at (3,0);
\coordinate (C) at (0,0);

\coordinate (X1) at (\x1,\y1);
\coordinate (X2) at (\x2,\y2);
\coordinate (X3) at (\x3,\y3);
\coordinate (X4) at (\x4,\y4);
\coordinate (X5) at (\x5,\y5);

\draw[thick] (A) arc[start angle=180,end angle=0,radius=3];
\draw[dotted] (-\r2,0) arc[start angle=180,end angle=0,radius=\r2];
\draw[dotted] (-\r5,0) arc[start angle=180,end angle=0,radius=\r5];

\node[point, label=left:$x_{\tau(i)}$] at (A) {};
\node[point, label=right:$x_{\tau(j)}$] at (B) {};
\node[point, label=below:$u_{i,j}$] at (C) {};
\node[point, label=$x_{k}$] at (X1) {};
\node[point, label=$x_{k+1}$] at (X2) {};
\node[point, label=below:$x_{n}$] at (X3) {};
\node[point, label=right:$x_{\tau(j) + 1}$] at (X4) {};
\node[point, label=right:$x_{{\tau(j)} + 2}$] at (X5) {};

\draw[dotted] (A) -- (B);
\draw[thick] (B) -- (X4);
\draw[thick] (X4) -- (X5);

\draw[thick] (X1) -- (X2);
\draw[thick,dotted] (X2) -- (X3);
\draw[thick,dotted] (X5) -- (X1);

\end{tikzpicture}
\caption{Geometric illustration of the balls
argument. For $0\le i<j\le q-2$, the semicircle represents the
upper half of the ball $\cK_{i,j}$ with diameter
endpoints $x_{\tau(i)}$ and $x_{\tau(j)}$ and center
$u_{i,j}$. Although the first iterate following
$x_{\tau(j)}$ may lie outside this ball, the two-step
decrease in distance to $u_{i,j}$ guarantees that
$x_{\tau(j)+2}$ lies in the interior of $\cK_{i,j}$. The one-step
decrease in distance then ensures that every following iterate is strictly closer to $u_{i,j}$ than the previous one, so that 
the last iterate
$x_n$ also lies in $\cK_{i,j}$ interior. This implies
$\ip{x_n-x_{\tau(i)}}{x_n-x_{\tau(j)}}<0$.
The whole argument requires at least one intermediate iterate between $x_{\tau(j)}$ and $x_{n}$, and leads to the bound $q -1 \le d$.}
\label{fig:circle}
\end{figure}
\end{proof}
\section{Lower Bound}\label{sec:lower-bound-proof}
The proof of \Cref{thm:global_minimax_lower_bound} crucially relies on the following technical result.
\begin{theorem}\label{thm:lb-d-dim}
Let $d \ge 3$, write $\numofpairs$ for the largest integer such that $2 \numofpairs + 1 \le d$ and let $n_0 = n_0(d) \coloneq 1 + (2/33) \cdot( 100^{\, \numofpairs} - 1)$. 
For every $\eta>0$ and every $n\ge n_0$, there exists $(f,g,\cX) \in \cI(d,L)$ and  $x_1\in B_{\star}(R;f, \cX)$ such that, for $x_n \coloneq x_{n}^{\rm sGM}(\eta, x_1; g,\cX)$,
\begin{align}\label{eq:lb-orbits}
    \Delta(x_n)
    \ge
    \dfrac{ d - 2}{8} \cdot \eta \cdot L^2\;.
\end{align}
\end{theorem}
\noindent We are now in the position to describe the desired lower bound for the minimax error.
\begin{proof}[Proof of Theorem~\ref{thm:global_minimax_lower_bound}]
The first part of the proof is showing that
\begin{equation}\label{eq:lb_minimax}
\forall d \ge 3, \quad \forall n \geq n_0(d), \quad \Delta_\star(d,L,n,R) \ge \frac{\sqrt{2}}{8} \cdot L \cdot R \cdot \sqrt{\frac{d-2}{n}}\;.    
\end{equation}
where $n_0(d)$ is defined as in the statement of Theorem~\ref{thm:lb-d-dim}. To see this, fix $\eta>0$. We make the following claim.
\begin{claim}\label{claim:lb-small-eta}
For every $d \ge 1$ there exist $(f,g_{f},\bbR^d) \in \cI(d,L)$ and $x_1\in B_\star(R; f, \bbR^d )$ such that
\[
\Delta\bigl(x_n^{\rm sGM}(\eta, x_1; g_{f}); f, \bbR^d \bigr)\ge \min\bigg\{\dfrac{L \cdot R}{2}, \dfrac{R^2}{4 \cdot n \cdot \eta }\bigg\}\;.
\]
\end{claim}
\noindent Combining Claim~\ref{claim:lb-small-eta} and Theorem~\ref{thm:lb-d-dim}  we get
\begin{equation}\label{eq:before-inf-eta}
\sup_{\substack{ (f, g_f, \cX)\in\cI(d, L)\\ x_1\in B_\star(R; f, \cX) }} \Delta\bigl(x_n^{\rm sGM}(\eta, x_1; g_f); f, \cX \bigr) \ge \max\bigg\{\frac{d - 2}{8} \cdot \eta \cdot L^2 , \min\bigg\{\frac{L \cdot R}{2}, \frac{R^2}{4 \cdot n \cdot \eta }\bigg\}\bigg\} \;,    
\end{equation}
and taking the infimum over all $\eta>0$ we get
\begin{equation*}
\Delta_\star(d,L,n,R) \ge \inf_{\eta>0} \max\bigg\{\dfrac{d - 2}{8} \cdot \eta \cdot L^2 , \min\bigg\{\dfrac{L \cdot R}{2}, \dfrac{R^2}{4 \cdot n \cdot \eta }\bigg\}\bigg\} = \frac{\sqrt{2}}{8} \cdot L \cdot R \cdot \sqrt{\frac{d-2}{n}}\;,
\end{equation*}
corresponding to the choice $\eta = \frac{R}{L} \cdot \sqrt{\frac{2}{n \cdot(d-2)}}$ that minimizes the right-hand side of \eqref{eq:before-inf-eta}.
Now that \eqref{eq:lb_minimax} is proved, we can proceed with the proof of the main theorem.
We first observe that
\begin{align}\label{eq:embedding}
\Delta_\star(d,L,n,R) \geq \Delta_\star(d',L,n,R) \quad \text{for every $1 \leq d' \leq d$.}
\end{align}
Indeed, an admissible instance in $\cI(d',L)$ can be embedded
into $\cI(d,L)$ by extending the objective to be constant in the
additional coordinates, extending the oracle subgradients to $d$-dimensional vectors by padding, and replacing
the feasible set $\mathcal{X}$ with
$\mathcal{X}\times\mathbb{R}^{d-d'}$. Initializing the additional
coordinates at zero preserves both the initial distance to the
minimizer set and the objective values along the trajectory, for
every constant stepsize.

\noindent
We next establish a lower bound in dimension one. Let $f(x):=L\cdot|x|$ for $x\in\mathbb{R}$, and define
\[
g(x)\coloneqq
\begin{cases}
L\cdot\operatorname{sign}(x), & x\neq0,\\
L/2, & x=0.
\end{cases}
\]
Then $(f,g,\mathbb{R})\in\mathcal{I}(1,L)$ and $0\in B_\star(R;f,\mathbb{R})$. For this instance, starting from $x_1=0$, the sGM iterates with constant stepsize $\eta>0$ alternate between $-\eta\cdot L/2$ and $\eta\cdot L/2$ after the first update. Hence,
\[
\Delta\!\left(x_n^{\mathrm{sGM}}(\eta,0;g,\mathbb{R});f,\mathbb{R}\right)
=\frac{\eta\cdot L^2}{2} 
\qquad\text{for every }n\geq2.
\]
Combining this instance with Claim~\ref{claim:lb-small-eta}, yields
\[
\Delta_\star(1,L,n,R)
\geq
\inf_{\eta>0}\max\left\{\frac{\eta\cdot L^2}{2},\min\left\{\frac{L\cdot R}{2},\frac{R^2}{4\cdot n\cdot\eta}\right\}\right\}
\geq\frac{L\cdot R}{\sqrt{8\cdot n}}.
\]
For the last inequality, if $\eta\geq R/(L\cdot\sqrt{2\cdot n})$, the first term in the maximum is at least $L\cdot R/\sqrt{8\cdot n}$; otherwise, both terms in the minimum are at least this quantity. By \eqref{eq:embedding} it holds that
\begin{equation}\label{eq:global_minimax_lower_bound_one_dimension}
\Delta_\star(d,L,n,R)
\geq
\frac{L\cdot R}{\sqrt{8\cdot n}} 
\qquad\text{for every }d\geq1\text{ and }n\geq2.
\end{equation}
This implies the claimed bound whenever $d\leq2$ or $n<100$, since then $\min\{d,\log n\}\leq\log100$.

\noindent
It remains to consider $d\geq3$ and $n\geq100$. Define the integer
\[
d'
\coloneqq
\min\left\{d,\, 2\left\lfloor\frac{\log n}{\log100}\right\rfloor+1\right\}.
\]
Then $3\leq d'\leq d$, and 
\[
n_0(d')
=
1+\frac{2}{33} \cdot \left(100^{\lfloor(d'-1)/2\rfloor}-1\right)
\leq 1+\frac{2}{33}\cdot(n-1)
\leq n.
\]
Moreover, by $d \ge 3$ we get $d-2\geq d/3\geq d/(2\cdot\log100)$, while by $n \ge 100$ we get
\[
2\left\lfloor\frac{\log n}{\log100}\right\rfloor-1
\geq \left\lfloor\frac{\log n}{\log100}\right\rfloor
\geq \frac{\log n}{2\cdot\log100}.
\]
Consequently,
\[
d'-2
=
\min\left\{d-2,\,2\left\lfloor\frac{\log n}{\log100}\right\rfloor-1\right\}\geq\frac{\min\{d,\log n\}}{2\log100}.
\]
Applying \eqref{eq:lb_minimax} in dimension $d'$ and using again \eqref{eq:embedding}, yields
\[
\begin{aligned}
\Delta_\star(d,L,n,R)
&\geq\Delta_\star(d',L,n,R)\\
&\geq\frac{\sqrt{2}}{8}\cdot L\cdot R\cdot\sqrt{\frac{d'-2}{n}}\\
&\geq\frac{1}{8\sqrt{\log100}}\cdot L\cdot R\cdot\sqrt{\frac{\min\{d,\log n\}}{n}},
\end{aligned}
\]
which completes the proof.
\end{proof}
\subsection{Proof of Theorem~\ref{thm:lb-d-dim}}
The proof relies on a suitable polyhedral function $\hat{f}$ defined as the maximum of $2\numofpairs+1$ linear functions. To study $\partial \hat{f}$, we will use a result of convex analysis, see, e.g., \cite[p.~199]{Bertsekas2009}: given a finite collection $\{f_j\}_{j  = 1}^N$ of convex functions defined on $\bbR^d$, if $F(x) \coloneq \max_{j = 1}^N f_j(x)$, then $F$ is convex; and given a function $f_{j_0}$ in the collection and some $x_0\in \bbR^d$, if the equality $F(x_0)=f_{j_0}(x_0)$ is attained, then $\partial F (x_0) \supseteq \partial f_{{j_0}} (x_0)$.
The definition of $\hat{f}$ is based on the following lemma:
\begin{lemma}\label{lem:ai-bi-def}
For every $1 \le i \le \numofpairs$, define the integer $m_i$ and $0<\epsilon_i \le 1/ 12$ by
\begin{equation}\label{eq:mi-epsi}    
 m_i \coloneq 3 \cdot 100^{\numofpairs -i}\;,\qquad  \epsilon_i \coloneq \frac{1}{4m_i}\;.
\end{equation}
There exist two sequences $(a_i)_{i=1}^{\numofpairs}$ and $(b_i)_{i=1}^{\numofpairs}$ of vectors in $\bbR^d$ such that:  1) the vectors $a_1$, $b_1$, $a_2$, $b_2$, $\dots$, $a_{\numofpairs}$, $b_{\numofpairs}$ are linearly independent; and 2) they satisfy
\begin{align*}
\ip{a_i}{a_i} &= 1 - 3 \epsilon_i\;,&
\ip{b_i}{b_i} &= 1\;,\quad &
\ip{a_i}{b_i} &= -(1 - 2 \epsilon_i) &
& & &\text{for}\quad 1\le i \le \numofpairs\;;\\
\ip{a_i}{a_j} &= \ip{a_i}{b_j}= 0\;, &
\ip{b_i}{a_j} &= \ip{b_i}{b_j}= -\epsilon_i &
& &
& & &\text{for}\quad  1\le i < j \le \numofpairs\;.
\end{align*}
\end{lemma}
\begin{proof}[Proof Idea.] It is possible to find $2\numofpairs$ linearly independent vectors $v_1$, $v_2$, $\dots$, $v_{2\numofpairs} \in \bbR^{d}$ with a prescribed Gram matrix $\mathbf{G} \in \bbR^{2\numofpairs \times 2 \numofpairs}$, i.e., such that $\mathbf{G}=(\ip{v_i}{v_j})_{i,j=1}^{2\numofpairs}$, provided two conditions are met: $2\numofpairs$ does not exceed $d$; and $\mathbf{G}$ is positive definite, see, e.g., \cite[p.~441]{Horn2013} for the case $2\numofpairs = d$. 
The desired conditions on the unknown vectors $(a_i)_{i=1}^{\numofpairs}$ and $(b_i)_{i=1}^{\numofpairs}$ can be used to build the desired Gram matrix $\mathbf{G}$ for the variables $v_1$, $v_2$, $\dots$, $v_{2\numofpairs}$ defined by $v_i\coloneq a_i$, $v_{\numofpairs + i} \coloneq \frac{1}{\sqrt{\epsilon_i}} \cdot (a_i + b_i)$ for $1 \le i \le \numofpairs$. All that remains is proving that $\mathbf{G}$ is positive definite. By construction, the off-diagonal entries of $\mathbf{G}$ are much smaller than those on the diagonal, which plays a role in the proof of positive definiteness, and which corresponds to the ``almost orthonormality'' of the variables $({v_i})_{i=1}^{2\numofpairs}$.
\end{proof}
\noindent Let $(m_i)_{i=1}^{\numofpairs}$, $(\epsilon_i)_{i=1}^{\numofpairs}$, $(a_i)_{i=1}^{\numofpairs}$ and $(b_i)_{i=1}^{\numofpairs}$  be as in Lemma~\ref{lem:ai-bi-def}.
Define
\begin{equation}
   \hat{f}(x) \coloneq \max\{  0 , \ip{a_i}{x}, \ip{b_i}{x} \colon  1\le i \le \numofpairs \}
\end{equation}
The function $\hat{f}$ is convex, nonnegative, $1$-homogeneous and it satisfies $\hat{f}(0)=0$, thus it is minimized at $x=0$. Moreover, noting that for every  $v \in \bbR^d$ the function $x \mapsto \ip{v}{x}$ is $\norm{v}$-Lipschitz on $\bbR^d$, it follows by Lemma~\ref{lem:ai-bi-def} that $\hat{f}$ is the maximum of $1$-Lipschitz functions, hence $\hat{f}$ is $1$-Lipschitz.\\
The rest of the theorem proof is divided into three cases associated with the conditions
\begin{enumerate*}[label=(C\arabic*)]
    \item $n = n_0$; and\label{const:n0}
    \item $L = \eta = 1$.\label{const:normalized}
\end{enumerate*}
First case: we assume both \ref{const:n0} and \ref{const:normalized} hold. We make the following claim:
\begin{claim}\label{claim:lb-base}
We have $n_0 = 1 + \sum_{i = 1}^{\numofpairs} 2m_i$. Furthermore, for every $1 \le k < n_0$ define $g_k$ by
\begin{equation}\label{eq:gk-in-claim}
 (g_k)_{k=1}^{n_0-1} \coloneq (\underbrace{a_1,b_1, \dots, a_1,b_1}_{\text{$m_1$ pairs}}, \underbrace{a_2,b_2, \dots, a_2,b_2}_{\text{$m_2$ pairs}}, \dots, \underbrace{a_{\numofpairs},b_{\numofpairs}, \dots, a_{\numofpairs},b_{\numofpairs}}_{\text{$m_{\numofpairs}$ pairs}})    
\end{equation}
and for every $1 \le k \le n_0$ define $x_k \coloneq -\sum_{i = 1}^{k-1} g_i$. Then there exists an oracle $\Gamma \in \cO_{\hat{f}}$ such that
$\Gamma(x_k)=g_k$ for every $1 \le k < n_0$, thus $(x_k)_{k=1}^{n_0}$ defines the first $n_0$ iterates of sGM initialized at $x_1=0$ with stepsize $\eta=1$ for the admissible problem instance $(\hat{f}, \Gamma, \bbR^{d}) \in \cI(d,1)$. Moreover,
\begin{equation}\label{eq:lb-equal}
    \hat{f}(x_{n_0}) - \hat{f}(0)= \hat{f}(x_{n_0}) =\dfrac{d_1}{4}\;.
\end{equation}
\end{claim}
\begin{proof}[Proof idea] Up to a suitable oracle choice, the behavior of sGM iterates initialized at $x_1=0$ with stepsize $\eta=1$ can be described in $\numofpairs$ stages.
The $i$-th stage takes $2m_i$ updates, where $(x_k,\hat{f}(x_k))$ climbs the graph of $\ip{a_i}{\cdot}$ for $k$ odd by bouncing on the graph of $\ip{b_i}{\cdot}$ for $k$ even, using  that $a_i$ and $b_i$ are ``nearly opposite'' by Lemma~\ref{lem:ai-bi-def}, but with $\norm{a_i} < \norm{b_i}$. Every stage increases the objective by $1/4$, so that the objective is $\numofpairs/4$ at the end of the last stage.
\end{proof}
\noindent
By Claim~\ref{claim:lb-base} and $\numofpairs \ge (d -2)/2$ we obtain \eqref{eq:lb-orbits} for $(f , g, \cX)=(\hat{f}, \Gamma, \bbR^d)$ and $x_1=0 \in \cX_{\star} \subseteq B_{\star}(R)$.\\
Second case: we assume \ref{const:n0} without assuming \ref{const:normalized}. Let $\eta > 0$, and let both $x_{n_0}$ and the oracle $\Gamma$  be as in Claim~\ref{claim:lb-base}. 
The function $L\hat{f}$ is convex and $L$-Lipschitz. By subdifferential calculus, $L\Gamma$ is a first-order oracle for $L\hat{f}$. By $1$-homogeneity of $\hat{f}$, a supporting hyperplane for $L\hat{f}$ at some $y\neq 0$ is also a supporting hyperplane at every positive multiple of $y$.
It follows that the function $g_{L,\eta} \colon \bbR^d\to \bbR^d$ defined by 
\[
g_{L,\eta}(x)\coloneq L\Gamma( x\eta^{-1}L^{-1})\quad\text{for every $x\in\bbR^d$}
\]
is also a first order oracle for $L\hat{f}$, and it satisfies the identity $\eta \cdot L \cdot  x - \eta \cdot g_{L,\eta}(\eta \cdot L \cdot x)= \eta \cdot L \cdot (x - \Gamma(x))$ by construction, which entails that $x'_{n_0}\coloneq x_{n_0}^{\rm sGM}(\eta, 0; g_{L,\eta},\bbR^d) =  L \cdot \eta \cdot x_{n_0}$. Consequently, for the admissible instance $(f , g, \cX)=(L \cdot \hat{f}, g_{L,\eta}, \bbR^d)$ we get $f(x'_{n_0}) = L \cdot \hat{f}(x'_{n_0}) = \eta \cdot L^2 \cdot \hat{f}(x_{n_0})$, and \eqref{eq:lb-orbits} follows by \eqref{eq:lb-equal}. Also in this case, $x_1=0 \in \cX_{\star} \subseteq B_{\star}(R)$.\\
Third case: $n > n_0$, i.e., we assume \ref{const:n0} does not hold. Call $\cW$ the linear span of $\{a_\ell, b_\ell \colon 1 \le \ell \le \numofpairs\}$, which is a $2 \cdot \numofpairs$-dimensional linear subspace of $\bbR^d$ that contains all the iterates described in the previous case. Since $2 \cdot \numofpairs<d$ by construction, there is one spare degree of freedom to modify the objective $L \cdot \hat{f}$ while preserving its behaviour on $\cW$.
Indeed, let $w_0\in \bbR^d$ be a unit vector that is orthogonal to $\cW$, and consider for every $\kappa>0$ the functions $f_{\kappa}(x) \coloneq \max\{ \kappa \cdot \ip{w_0}{x}, L \cdot \hat{f}(x)\}$. 
By construction $f_{\kappa}(x) = L \cdot \hat{f}(x)$ for every $x \in \cW$, we deduce that for every  $x \in \cW$ the oracle $g_{L,\eta}$ queried in $x$ returns an element of $\partial f_{\kappa}(x)$.
Also note that $\hat{f}(t w_0) = 0$ for every $t>0$ by definition of $w_0$, hence $f_{\kappa}$ coincides with $\kappa \cdot \ip{w_0}{\cdot}$ on the semiline generated by $w_0$, thus $\kappa \cdot w_0 \in \partial f_{\kappa}(tw_0)$ for every $t >0$.
It is therefore possible to consider a first-order oracle $g_{f_\kappa}$ for $f_\kappa$ such that
\begin{enumerate}
    \item $g_{f_\kappa}(x)=g_{L,\eta}(x)$ for every $x\in \cW$;
    \item $g_{f_\kappa}(tw_0)=\kappa \cdot w_0 $ for every $t>0$.
\end{enumerate}
Let $\eta>0$, define $x_1 \coloneq (n - n_0) \cdot \eta \cdot \kappa \cdot w_0$ and consider $x_k \coloneq x_k^{\rm sGM}(\eta, x_1; g_{f_{\kappa}}, \bbR^d)$ for every $1 \le k\le n$.
By induction, we see that  $x_{k} = (n - n_0 - k +1) \cdot \eta \cdot \kappa \cdot w_0$ for every $1 \le k\le n-n_0+1$, so that  $x_{n - n_0 + 1} = 0$. But now, the remaining $n_0-1$ updates reproduce the same behavior of sGM established in the second proof case, and we get \eqref{eq:lb-orbits}, since $x_{n-n_0+k}^{\rm sGM}(\eta, x_1; g_{f_{\kappa}}, \bbR^d) = x_k^{\rm sGM}(\eta, 0; g_{L,\eta}, \bbR^d)$ for $1 \le k \le n_0$.
What is left is finding a value of $\kappa$ for which $(f, g, \cX)=(f_{\kappa}, g_{f_{\kappa}}, \bbR^d)$ is an admissible instance and $x_1\in B_{\star}(R;f, \cX)$. The Lipschitz constant of $f_{\kappa}$ is $\max\{\kappa, L\}$ and
$\dist(x_1, \cX_{\star}) = (n - n_0) \cdot \eta \cdot \kappa$. It is hence sufficient to choose $\kappa \coloneq \min\{L, R/[(n - n_0) \cdot \eta]\}$, for which $f_\kappa$ is $L$-Lipschitz and $\dist(x_1, \cX_{\star})\le R$.\\
\section{Conclusions}\label{sec:conclusions}
\noindent
We have characterized, up to universal numerical constants, the minimax last-iterate error of projected sGM with a constant stepsize in terms of all problem parameters $d,L,n,R$, without requiring a bounded feasible set. In every fixed dimension, the classical $1/\sqrt{n}$ dependence on the horizon $n$ holds without extra logarithmic factors in $n$, even with the standard stepsize. Optimizing the constant stepsize improves the dimension dependence from $d$ to $\sqrt{d}$, with both rates saturating at their respective dimension-free counterparts.

\noindent
Determining the optimal numerical constants remains an open question. A broader direction is to complement this worst-case characterization with an instance-dependent analysis. Such an analysis could identify geometric features of the objective and feasible set that allow sharper last-iterate guarantees than those determined by the global problem parameters alone.

\appendix
\section{Remaining Proofs for Section~\ref{sec:upper-bound-proof}}
\subsection{Proof of Theorem~\ref{thm:upper-bound}}\label{sec:upper-thm-proof}
Let $k_{\star}$ be the maximum integer in $\argmin_{ 1 \le k  \le n} f(x_k)$. We have
\[
f(x_n) - f_{\star} = \underbrace{f(x_{k_\star}) - f_\star}_{\eqcolon A} + \underbrace{f(x_n) - f(x_{k_\star})}_{\eqcolon B}\,.
\]
To bound $A$, we can apply the classical Theorem~\ref{thm:best-iterate}.
To bound $B$, note that at every iteration
\begin{equation}\label{eq:onestep-etaL2}
\abs{f(x_{k+1}) - f(x_k)} \markrel{rel:lip}{\le}  L\norm{x_{k+1} - x_k} \markrel{rel:nonexp}{\le} L \norm{ \eta g_k} \markrel{rel:subg-norm}{\le} \eta L^2,
\end{equation}
where: \relref{rel:lip} holds by Lipschitzness of $f$; \relref{rel:nonexp} holds by non-expansiveness of the projection; \relref{rel:subg-norm} holds by bounding the subgradient norm with the Lipschitz constant $L$.
We will show that $B$ admits a control in terms of $\eta L^2$ too; more specifically, we will show that
\begin{equation}\label{eq:upper-trail}
B = f(x_n) - f(x_{k_\star})
\le
d \eta L^2 \,,
\end{equation}
which together with Theorem~\ref{thm:best-iterate} gives \eqref{eq:upper-bound}.
We first treat the case $d=1$, where we prove \eqref{eq:upper-trail} as a consequence of an argument mentioned in \cite{Koren2020}. For $d=1$, the condition $f(x_{k+1})>f(x_{k})$ is possible only if $x_k$ lies in some neighborhood of radius $\eta L$ about a point $x_{\star} \in \cX_{\star}$, and the following iterates cannot escape from such neighborhood, which is mapped by $f$ into the interval $[f_{\star}, \eta L^2 + f_{\star}]$. It follows that if $f(x_n)>f(x_{k_\star})$, then $f(x_n)$ and $f(x_{k_\star})$ lie in a common interval of length $\eta L^2$, which proves \eqref{eq:upper-trail}. Our proof can now focus on the case $d \ge 2$. Define $q$ as the minimum integer such that $f(x_n) - f(x_{k_\star}) \le q \eta L^2$ and assume $q \ge 2$, for otherwise \eqref{eq:upper-trail} is already proved. The proof is complete if we show that $q \le d$. Define the bins $(\cB(i))_{i \in \bbZ}$ by
\[
\cB(i) \coloneq ((i-1) \eta L^2 + f(x_{k_\star}),  i \eta L^2 + f(x_{k_\star})]
\]
and observe that, by~\eqref{eq:onestep-etaL2},
\begin{equation}\label{eq:bin-jumps}
f(x_k) \in \cB(i) \quad \implies \quad f(x_{k+1}) \in \cB({i-1}) \cup \cB({i}) \cup \cB({i+1}).
\end{equation}
In particular, since the sequence $\big(f(x_{k})\big)_{k = k_{\star}}^n$ starts in $\cB(0)$ and ends in $\cB(q)$, we deduce that, however we choose $0 \le i \le q$, 
we get $\cB(i) \ni f(x_k)$ for at least one $k_\star \le k \le n$; it is then well defined the sequence of integers $(\tau(i))_{i = 0}^q$ of length $q+1$ given by
\[
\tau(i) \coloneq \max\{ k \in \bbN \;|\; k_{\star} \le  k \le n, \; f(x_k) \in \cB(i) \}\,.
\]
Note that the maximality in the definition of $k_\star$ gives $\tau(0) = k_{\star}$.
Also, by $f(x_n) \in \cB(q)$ we get $\tau(q) = n$. We can show that for every $0 \le i < q$,
\begin{align}
f(x_{\tau(i)}) &\le i\cdot \eta \cdot L^2 + f(x_{k_{\star}})\,, \label{eq:in-tau}
\\
f(x_{k}) &> i\cdot \eta \cdot L^2 + f(x_{k_{\star}}) \qquad \text{for $\tau(i) < k \le n$}\,,\label{eq:after-tau}
\end{align}
which combined together give
\begin{equation}\label{eq:after-tau-weak}
f(x_{k})  > f(x_{\tau(i)}) \qquad \text{for $\tau(i) < k \le n$}\,.
\end{equation}
The proof of \eqref{eq:in-tau} follows by definition of $\tau$, while 
\eqref{eq:after-tau} is a consequence of 
\eqref{eq:bin-jumps}, by which the sequence $(f(x_k))_{\tau(i)+1}^n$ takes values in the connected component of $\bbR\setminus \cB(i)$ that contains $\cB(q)$, i.e., in the interval $(i\cdot \eta \cdot L^2 + f(x_{k_{\star}}), +\infty)$.
In particular, by \eqref{eq:after-tau} we deduce that $\tau$ is increasing.

\noindent For every choice of integers $0 \le i < j \le q-1$, define 
\begin{equation}\label{eq:fuij-def}
u_{i,j} \coloneq \frac{x_{\tau(i)} + x_{\tau(j)}}{2}\,.
\end{equation}
We are going to study how the iterates $x_k$ vary their distance from every point $u_{i,j}$ for sufficiently large $k$.
To this end, we first show the following inequalities:
\begin{align}
f(u_{i,j}) 
&\le 
\frac{f(x_{\tau(i)}) + f(x_{\tau(j)})}{2} \,,\label{eq:fuij-conv}
\\
f(u_{i,j}) 
&\le 
\bigg( j - \frac{\abs{j - i}}{2} \bigg) \cdot \eta \cdot L^2  + f(x_{k_{\star}})\,.\label{eq:fuij-tau}
\end{align}
Note that \eqref{eq:fuij-conv} holds by convexity of $f$, whereas \eqref{eq:fuij-tau} follows by  \eqref{eq:fuij-conv} after noting that
\[
 \frac{f(x_{\tau(i)}) + f(x_{\tau(j)})}{2}
 \markrel{sym:I}{\le}
 \frac{i + j}{2}\cdot \eta \cdot L^2 + f(x_{k_{\star}})
 \markrel{sym:II}{=}
 \bigg( j - \frac{\abs{j - i}}{2} \bigg) \cdot \eta \cdot L^2  + f(x_{k_{\star}}) \,,
\] 
where: \relref{sym:I} follows by \eqref{eq:in-tau}; \relref{sym:II} by $i<j$. These inequalities will be used in conjunction with Lemma~\ref{lem:1-2step}. More specifically,
by Lemma~\ref{lem:1-2step}~\ref{it:1step}, every condition on $i$, $j$, $k$ such that  $f(x_{k}) >  f(u_{i,j}) + \frac{1}{2} \cdot \eta \cdot L^2\,$
implies that
\begin{equation}\label{eq:go-closer-to-uij}
    \norm{x_{k+1} - u_{i,j}} < \norm{x_{k} - u_{i,j}}\,.
\end{equation}
To find such a condition, %
we apply \eqref{eq:fuij-tau} to get
\begin{equation}\label{eq:fuij-tauonehalf}
 f(u_{i,j}) + \frac{1}{2} \cdot \eta \cdot L^2 \le \underbrace{\bigg( j - \frac{\abs{j - i}-1}{2} \bigg)}_{\le  j} \cdot \eta \cdot L^2 + f(x_{k_\star})\,.    
\end{equation}
If $\tau(j)< k \le n$, then  $f(x_{k})  >  j \cdot \eta \cdot L^2 + f(x_{k_\star})$ by \eqref{eq:after-tau},
and so by \eqref{eq:fuij-tauonehalf} we see that%
\footnote{In \eqref{eq:overonehalf}, the case $k=n$ will be useful in the proof of Lemma~\ref{lem:virtual}.}
\begin{equation}\label{eq:overonehalf}
 f(x_{k}) >  f(u_{i,j}) + \frac{1}{2} \cdot \eta \cdot L^2 \quad \text{for $\tau(j)< k \le n$}\,,
\end{equation}
which  by  Lemma~\ref{lem:1-2step}~\ref{it:1step} gives \eqref{eq:go-closer-to-uij} for $\tau(j)< k < n$.
In particular, this means that
\begin{equation}\label{eq:distance-decreases}
(\norm{x_{k} - u_{i,j}})_{k = \tau(j) + 1}^n \qquad \text{is a decreasing sequence,}
\end{equation}
which we will now use to show that
\begin{equation}\label{eq:bound-on-xn}    
\norm{x_n - u_{i,j} } < \norm{x_{\tau(j)} - u_{i,j}} \quad  \text{for} \quad \tau(j) \le n -2.
\end{equation}
Indeed, assume $\tau(j) \le n - 2$. To apply Lemma~\ref{lem:1-2step}~\ref{it:2step} to bound $\norm{x_{\tau(j) + 2} - u_{i,j}}$, note that
\begin{align*}    
    2 \cdot \eta^2 L^2
-
    2\cdot \eta
\cdot \sum_{k = 1}^2
    \big( f(x_{\tau(j) + k-1})  -  f(u_{i,j}) \big)
&=
    2 \cdot \eta
\cdot
    \big(
    \eta L^2 
+ 
    2f(u_{i,j}) - f(x_{\tau(j)}) - f(x_{\tau(j) + 1})
    \big)\\
&\markrel{rel:hpmet1}{\le}
    2 \cdot \eta
\cdot
    \big(
    \eta L^2 + f(x_{\tau(i)}) - f(x_{\tau(j) + 1}) 
    \big)\\
&\markrel{rel:hpmet2}{\le}
    2 \cdot \eta
\cdot
    \big(
    (i + 1) \eta L^2 + f(x_{k_{\star}}) -  f(x_{\tau(j) + 1}) 
    \big)\\
&\markrel{rel:hpmet3}{<}
    2 \cdot \eta
\cdot
    (i + 1- j ) \cdot \eta L^2 \le    0 \,,
\end{align*}
where: \relref{rel:hpmet1} holds by \eqref{eq:fuij-conv}; \relref{rel:hpmet2} by \eqref{eq:in-tau}; the strict inequality \relref{rel:hpmet3} by \eqref{eq:after-tau}.
Lemma~\ref{lem:1-2step}~\ref{it:2step} gives
\[
\norm{x_{\tau(j) + 2} - u_{i,j}}< \norm{x_{\tau(j)} - u_{i,j}}\,,
\]
which together with \eqref{eq:distance-decreases} proves \eqref{eq:bound-on-xn}.\\
\noindent What we have proved so far can be used to deduce that $q \le d + 1$, i.e., by definition of $q$, that
\begin{equation}\label{eq:d+1-bound}
f(x_n) - f(x_{k_\star})
\le
(d+1) \eta L^2 \,.
\end{equation}
Indeed, we have 
\begin{subequations}\label{eq:conditions-on-xn}
        \begin{empheq}[left={\empheqlbrace}]{alignat = 2}
        \norm{x_{n} - u_{i,j} } &< \norm{x_{\tau(j)} - u_{i,j}}  & \qquad&\text{for every } \quad 0 \le i< j \le q - 2\,,
        \label{eq:bound-on-xn-sys}\\
        f(x_{n}) &> f(x_{\tau(\ell)}) & &\text{for every } \quad 0 \le \ell \le q - 2\,,
        \label{eq:highfxn}
        \end{empheq}
\end{subequations}
where \eqref{eq:bound-on-xn-sys} is a consequence of \eqref{eq:bound-on-xn}, since  $\tau(j)\le n-2$ is automatically verified for $j \le q - 2$, whereas  \eqref{eq:highfxn} follows by $n = \tau(q)>\tau(i)$ and \eqref{eq:after-tau-weak}.
Observe that
\begin{align*}
\ip{x - x_{\tau(i)} }{x - x_{\tau(j)}}
&\markrel{rel:pol}{=} \dfrac{1}{4} \Big( \norm{(x - x_{\tau(i)}) + (x - x_{\tau(j)})}^2 - \norm{(x - x_{\tau(i)}) - (x - x_{\tau(j)})}^2 \Big)\\
&\markrel{rel:radius}{=} \norm{x - u_{i,j}}^2 - \norm{x_{\tau(j)} - u_{i,j}}^2 \,.  
\end{align*}
where: \relref{rel:pol} holds by the polarization identity; \relref{rel:radius} holds by $\norm{x_{\tau(j)}- u_{i,j}} = \frac{1}{2}\norm{x_{\tau(i)}- x_{\tau(j)}}$. In particular, defining  $w_\ell\coloneq x_{n} - x_{\tau(\ell)}$ for every $0\le \ell \le q-2$, we deduce from  \eqref{eq:bound-on-xn-sys} that
\[
\ip{w_i}{w_j}<0 \qquad \text{for every } \quad 0 \le i< j \le q - 2\,.
\]
Moreover, choose $g_n \in \partial f (x_n)$, which by definition of subgradient gives
\begin{equation}\label{eq:subgr-hn}
\ip{g_n}{x_n - y} \ge f(x_n) - f(y) \qquad \text{for every } y \in \bbR^d \,.
\end{equation}
Combining \eqref{eq:subgr-hn} with \eqref{eq:highfxn} we obtain
\[
\ip{g_n}{w_\ell}>0  \qquad \text{for every }\quad 0 \le \ell \le q - 2\,.
\]
By Lemma~\ref{lem:rank}, the vectors $w_0$, $w_1$, $\dots$, $w_{q-2}$ are linearly independent. But since they are elements of $\bbR^d$, their number $q -1$ cannot exceed $d$, i.e., that $q \le d + 1$, which proves \eqref{eq:d+1-bound}.\\
\noindent 
To improve \eqref{eq:d+1-bound}, we adapt the same idea to apply Lemma~\ref{lem:rank} to $q$ vectors. This leads us to consider a system similar to \eqref{eq:conditions-on-xn}, but where $x_n$ is replaced by a suitable subsequent iterate obtained with a step size possibly different from $\eta$. Fix $g_n \in \partial f(x_n)$ and, for every $\alpha \ge 0$, define $\hat{x}_{n+1}(\alpha)$ by 
\begin{equation}
    \hat{x}_{n+1}(\alpha) \coloneq \Pi_{\cX}(x_n - \alpha \cdot g_n).
\end{equation}
\begin{lemma}\label{lem:virtual}
There exists $\alpha \ge 0$ such that
\begin{subequations}\label{eq:conditions}
        \begin{empheq}[left={\empheqlbrace}]{alignat = 3}
        \norm{\hat{x}_{n+1}(\alpha) - u_{i,j} } &< \norm{x_{\tau(j)} - u_{i,j}}\,, & \qquad&\text{for every } \quad 0 \le i< j \le q - 1\,,
        \label{eq:bound-on-virtual}\\
        f(\hat{x}_{n+1}(\alpha)) &> f(x_{\tau(\ell)})& &\text{for every } \quad 0 \le \ell \le q - 1\,.
        \label{eq:high}
        \end{empheq}
\end{subequations}
\end{lemma}
\noindent Let $\alpha$ be as in Lemma~\ref{lem:virtual}. For every $0 \le i \le q-1$, define the vector $v_i \coloneq \hat{x}_{n+1}(\alpha) - x_{\tau(i)}$. Then, arguing as before, $\ip{v_i}{v_j}<0$ for every $ 0 \le i< j \le q - 1$ by \eqref{eq:bound-on-virtual} and given $\hat{g}_{n+1}(\alpha) \in \partial f(\hat{x}_{n+1}(\alpha))$, the subgradient inequality together with \eqref{eq:high} gives  $\ip{v_\ell}{\hat{g}_{n+1}(\alpha)}>0$ for every $ 0 \le \ell \le q - 1$. By Lemma~\ref{lem:rank}, the vectors  $v_0$, $v_1$, $\dots$, $v_{q-1}$ are linearly independent, thus their number $q$ cannot exceed $d$. This proves \eqref{eq:upper-trail}.
\subsection{Proof of Lemma~\ref{lem:1-2step}}
\begin{proof}[Proof of Lemma~\ref{lem:1-2step}]
We recall the following classical inequality, (sometimes called the fundamental inequality of the subgradient method) valid for every $x$, $u \in \cX$, every $g \in \partial f (x)$ and every $\alpha \ge 0$:
\[
\norm{\Pi_{\cX}(x - \alpha \cdot g) - u}^{2}
\le
\norm{x - u}^{2}
+
\alpha^2
\norm{g}^{2}
-
2\cdot \alpha
\cdot
(f(x)  - f(u))\,.
\]
Then, by $\norm{g'_k} \le L$
for $k \in \{1, 2\}$ we get
\begin{equation}\label{eq:fundkL}
\norm{y_{k+1} - u}^{2}
-
\norm{y_k - u}^{2}
\le
\alpha_k^2
L^{2}
-
2\cdot \alpha_k
\cdot
(f(y_k)  - f(u)) \,.
\end{equation}
\ref{it:1step}: A sufficient condition for $\norm{y_{2} - u} < \norm{y_1 - u}$ is that the right-hand side of \eqref{eq:fundkL} for $k = 1$ is negative, i.e., that 
$\alpha_1 \cdot \big( \alpha_1 \cdot L^2 - 2 \cdot (f(y_1) - f(u))\big) < 0$, which solved for $f(y_1)$ proves the claimed implication.\\
\noindent \ref{it:2step}: Summing termwise \eqref{eq:fundkL} for $k \in \{1, 2\}$, we get
\begin{equation}\label{eq:telescoped}
\norm{y_{3} - u}^{2}
-
\norm{y_1 - u}^{2}
\le
\sum_{k = 1}^2 \alpha_k^2
L^{2}
-
2\cdot \sum_{k = 1}^2 \alpha_k
\cdot
(f(y_k)  - f(u))\,.
\end{equation}
In particular, a sufficient condition for $\norm{y_{3} - u} < \norm{y_1 - u}$ is that the right-hand side of \eqref{eq:telescoped} is negative, and this proves the claimed implication.
\end{proof}
\subsection{Proof of Lemma~\ref{lem:virtual}}
\begin{proof}[Proof of Lemma~\ref{lem:virtual}]
Let $0 \le i< j \le q - 1$ and $0\le \ell \le q-1$ and consider the inequalities
\begin{align}
\norm{\hat{x}_{n+1}(\alpha) - u_{i,j} } &< \norm{x_{\tau(j)} - u_{i,j}}\,,
\label{eq:bound-on-virtual-proof}\\
f(\hat{x}_{n+1}(\alpha)) &> f(x_{\tau(\ell)}) \,.
\label{eq:high-proof}
\end{align}
We first show that the condition
\begin{equation}\label{eq:high++}
 0 < \alpha < \frac{f(x_n) - f(x_{\tau(q-1)})}{L^2}   
\end{equation}
gives \eqref{eq:high-proof} for every choice of $0\le \ell \le q-1$.
Indeed, note that
\[
\abs{f(x_n) - f(\hat{x}_{n+1}(\alpha))}
\le L \norm{x_n - \hat{x}_{n+1}(\alpha)} \le L \norm{\alpha \cdot g_n} \le \alpha \cdot L^2  
\]
and so 
\begin{equation}\label{eq:lb-on-fxn+1}
f(\hat{x}_{n+1}(\alpha)) 
=
f(x_n) + f(\hat{x}_{n+1}(\alpha)) - f(x_n)
\ge 
f(x_n) - \alpha \cdot L^2\,.  
\end{equation}
If \eqref{eq:high++} holds, then
\[
 f(x_n) - \alpha \cdot L^2  > f(x_{\tau(q -1)}) 
\ge f(x_{\tau(\ell)})\,,
\]
which would give \eqref{eq:high-proof} by \eqref{eq:lb-on-fxn+1}.\\
\noindent
The proof proceeds by cases to deduce restrictions on $\alpha$ that are stronger than \eqref{eq:high++} and that imply \eqref{eq:bound-on-virtual-proof}. We begin with the case $\tau(j) \le n-2$. In this case, 
note that by Lemma~\ref{lem:1-2step}~\ref{it:1step}, we get
\begin{align*}
f(x_{n}) &> f(u_{i,j}) +  \frac{\alpha}{2} \cdot L^2  & & \implies & \norm{\hat{x}_{n+1}(\alpha) - u_{i,j}} &< \norm{x_{n} - u_{i,j}},
\end{align*}
and so by \eqref{eq:overonehalf} restricted to $k = n$ we observe that 
$0 < \alpha \le \eta$ gives
$\norm{\hat{x}_{n+1}(\alpha) - u_{i,j}} < \norm{x_{n} - u_{i,j}}$.
In particular, by \eqref{eq:bound-on-xn}, we see that requiring
\begin{equation}\label{eq:general-case-requirement}
 0 < \alpha< \min\bigg\{ \frac{f(x_n) - f(x_{\tau(q-1)})}{L^2} ,\eta \bigg\}   
\end{equation}
is a condition that is independent of $i$ and $j$, that is stronger than \eqref{eq:high++} and that gives \eqref{eq:bound-on-virtual-proof} in the case $\tau(j) \le n-2$.\\
\noindent We now treat the case $\tau(j) > n-2$. We recall that $\tau(j) \le \tau(q-1) < \tau(q) = n$ by strict monotonicity of $\tau$, and so if this case occurs then we must have $\tau(j) = n-1$ and $j = q-1$.
In particular,
\[
 0 < \frac{f(x_n) - f(x_{\tau(q-1)})}{L^2} = \frac{f(x_n) - f(x_{n-1})}{L^2} \le \eta   
\]
by \eqref{eq:onestep-etaL2}, hence \eqref{eq:general-case-requirement} and \eqref{eq:high++} become equivalent conditions on $\alpha$.\\
Let $\Psi(\alpha)$ be the polynomial in $\alpha$ defined by
\[
\Psi(\alpha) \coloneq
\eta^2 L^2 - \eta \cdot(f(x_{\tau(q-1)}) - f(x_{\tau(q-2)})) 
+ \alpha\cdot \Big(\alpha \cdot L^2 - (2 f(x_n) - f(x_{\tau(q-1)}) - f(x_{\tau(q-2)})\Big)\,.
\]
We observe that
\begin{equation}\label{eq:gap-n-ti}
f(x_{n})- f(x_{\tau(q-2)}) > \eta L^2
\end{equation}
and a direct computation shows that
\begin{align*}    
\Psi\bigg( &\frac{f(x_n) - f(x_{\tau(q-1)})}{L^2}\bigg)
=
\eta^2 L^2 - \eta \cdot(f(x_{\tau(q-1)}) - f(x_{\tau(q - 2)}))\\ 
&\quad\qquad+ \frac{f(x_n) - f(x_{\tau(q-1)})}{L^2} \cdot \Big( {f(x_n) - f(x_{\tau(q-1)})} - \big(2 f(x_n) - f(x_{\tau(q-1)}) - f(x_{\tau(q - 2)})\big)\Big)\\
&\qquad=
\eta^2 L^2 - \eta \cdot(f(x_{\tau(q-1)}) - f(x_{\tau(q - 2)}))  - \frac{f(x_n) - f(x_{\tau(q-1)})}{L^2} \cdot \big( f(x_n) - f(x_{\tau(q - 2)})\big)\\
&\qquad\markrel{rel:i-is-small}{<}
\eta^2 L^2 - \eta \cdot(f(x_{\tau(q-1)}) - f(x_{\tau(q - 2)}))  - \frac{f(x_n) - f(x_{\tau(q-1)})}{L^2} \cdot \eta \cdot L^2\\
&\qquad=
\eta^2 L^2 - \eta \big( f(x_n) - f(x_{\tau(q - 2)})\big)\\
&\qquad\markrel{rel:i-is-small-bis}{<}
0
\end{align*}
where in \relref{rel:i-is-small} and \relref{rel:i-is-small-bis} we used \eqref{eq:gap-n-ti}. By continuity of $\Psi$, there exists $\alpha_0$ subject to \eqref{eq:high++} such that $\Psi(\alpha_0)<0$.\\
By Lemma~\ref{lem:1-2step}~\ref{it:2step}, we obtain \eqref{eq:bound-on-virtual-proof} if we show that the quantity
\begin{align*}
\eta^2 L^2 + \alpha^2 L^2 - 2 \cdot \alpha \cdot (f(x_n) - f(u_{i,j})) - 2 \cdot \eta \cdot(f(x_{n-1}) - f(u_{i,j}))
\end{align*}
is negative. Note that for every $\alpha \ge 0$
\begin{align*}
&\eta^2 L^2 + \alpha^2 L^2 - 2 \cdot \alpha \cdot (f(x_n) - f(u_{i,j})) - 2 \cdot \eta \cdot(f(x_{n-1}) - f(u_{i,j}))\\
&\markrel{rel:conv}{\le}     \eta^2 L^2 + \alpha^2 L^2 - \alpha \cdot (2 f(x_n) - f(x_{n-1}) - f(x_{\tau(i)})) - \eta \cdot(f(x_{n-1}) - f(x_{\tau(i)})),\\
&\markrel{rel:ub-for-i}{\le}     \eta^2 L^2 + \alpha^2 L^2 - \alpha \cdot (2 f(x_n) - f(x_{n-1}) - f(x_{\tau(q-2)})) - \eta \cdot(f(x_{n-1}) - f(x_{\tau(q-2)})),\\
&\markrel{rel:n-1=tq-1}{=}  \Psi(\alpha),
\end{align*}
where: \relref{rel:conv} holds by \eqref{eq:fuij-conv}; \relref{rel:ub-for-i} by $ f(x_{\tau(i)}) \le  f(x_{\tau(q-2)})$; \relref{rel:n-1=tq-1} by $ f(x_{n-1}) = f(x_{\tau(q-1)})$.
In particular, for $\alpha = \alpha_0$, we can apply Lemma~\ref{lem:1-2step}~\ref{it:2step} to get \eqref{eq:bound-on-virtual-proof}.
\end{proof}
\subsection{Proof of Lemma~\ref{lem:rank}}\label{sec:rank}
For completeness, we report here a proof of the classical Lemma~\ref{lem:rank}.
\begin{proof}[Proof of Lemma~\ref{lem:rank}]
Suppose there exist some coefficients $\alpha_1$, $\dots$, $\alpha_M$, such that $\sum_{k=1}^M \alpha_k \cdot v_k = 0$, and define the sets of integers
\[
I^{+}\coloneq\{1 \le  k  \le M \mid \alpha_k > 0\},
\qquad
I^{-}\coloneq\{1 \le  k  \le M \mid \alpha_k < 0\}.
\]
Suppose, for a contradiction, that either $I^{+}$ or $I^{-}$ is nonempty. Then by
\[
0
=
\ip{\sum_{k=1}^M \alpha_k \cdot  v_k}{u}
=
\sum_{k\in I^{+}} \alpha_k \cdot  \ip{v_k}{u}
+
\sum_{k\in I^{-}} \alpha_k \cdot  \ip{v_k}{u}
\]
and by the hypothesis on $u$ we deduce that both $I^{+}$ and $I^{-}$ are nonempty.
But then, the vector
\[
w
\coloneq
\sum_{k\in I^{+}}\alpha_k \cdot v_k
=
\sum_{\ell \in I^{-}}(-\alpha_\ell) \cdot v_\ell.
\]
satisfies
\[
\norm{w}^2
=
\ip{\sum_{k\in I^{+}} \alpha_k \cdot  v_k}{\sum_{\ell\in I^{-}} (-\alpha_\ell) \cdot  v_\ell}
=
\sum_{k\in I^{+}}\sum_{\ell \in I^{-}} \alpha_k \cdot (-\alpha_\ell) \cdot \ip{v_k}{v_\ell}
\le
0,
\]
i.e., $w = 0$, even though
\[
\ip{w}{u}= \sum_{k\in I^{+}}\alpha_k \cdot \ip{v_k}{u} > 0\;, 
\]
which is impossible.
This means that both $I^{+}$ and $I^{-}$ are empty, i.e., that the coefficients $\alpha_1$, $\dots$, $\alpha_M$ must be all zeros.
\end{proof}
\section{Remaining Proofs for Section~\ref{sec:lower-bound-proof}}
\noindent In this section, we describe the postponed proofs:
\subsection{Proof of Lemma~\ref{lem:ai-bi-def}}
\begin{proof}[Proof of Lemma~\ref{lem:ai-bi-def}]
By construction, $(\epsilon_i)_{i = 1}^{\numofpairs}$ is an increasing geometric progression with common ratio $100$. In particular,  $\epsilon_i \le \epsilon_{\numofpairs}= 1/12$ for every $1 \le i \le \numofpairs$.
Consider the symmetric $2\numofpairs\times 2 \numofpairs$ matrix $\mathbf{G}$ defined by
\begin{equation}
\mathbf{G} \coloneq
\begin{pmatrix}
    \mathbf{D} & \mathbf{E}\\
    \mathbf{E}^{\top} &\mathbf{R}
\end{pmatrix}\;,
\end{equation}
where $\mathbf{D}$, $\mathbf{E}$ and $\mathbf{R}$ are $\numofpairs \times \numofpairs$ matrices defined as follows:
\begin{equation}
    \mathbf{D}\coloneq \diag((1 -3\epsilon_i)_{i={1}}^{\numofpairs})\;,
\qquad
    \mathbf{E}\coloneq
    -\begin{pmatrix}
    \sqrt\epsilon_1 & 0          &  0            & \dots     & 0         &0\\
    \sqrt\epsilon_1 & \sqrt\epsilon_2 &  0            & \dots     & 0         &0\\
    \sqrt\epsilon_1 & \sqrt\epsilon_2 &  \sqrt\epsilon_3   & \dots     & 0         &0\\
    \vdots     & \vdots     &   \vdots      & \ddots    & \vdots    &\vdots   \\
    \sqrt\epsilon_1 & \sqrt\epsilon_2 &  \sqrt\epsilon_3   & \dots     & \sqrt\epsilon_{\numofpairs-1} &0\\
    \sqrt\epsilon_1 & \sqrt\epsilon_2 &  \sqrt\epsilon_3   & \dots     & \sqrt\epsilon_{\numofpairs-1} &\sqrt\epsilon_{\numofpairs}
    \end{pmatrix}\;,
\end{equation}    
\begin{equation}
    R_{i,j} \coloneq \begin{cases}
        1 &\text{for $i = j$,}\\
        -2\cdot 10^{-\abs{i-j}} &\text{for $i \neq j$.}\\
    \end{cases}
\end{equation}
We shall prove that $\mathbf{G}$ is positive definite. 
The minimum eigenvalue of the diagonal matrix $\mathbf{D}$ is $(1 - 3\epsilon_{\numofpairs})$, hence $x^\top\mathbf{D}x \ge (1 - 3\epsilon_{\numofpairs}) \norm{x}^{2}$ for every $x \in \bbR^{\numofpairs}$. Regarding $\mathbf{R}$, every element on the main diagonal equals $1$, and for every $1\le i \le \numofpairs$, the $i$-th row of $\mathbf{R}$ has sum of the magnitudes of the off-diagonal entries that can be bounded as follows:
\begin{equation}
\sum_{j \ne i} \abs{R_{i,j}} = \sum_{j \ne i} \abs*{-2 \cdot 10^{-\abs{i-j}}} \le  4\cdot \sum_{j = 1}^{\infty} 10^{-j} = \dfrac{4}{9} \;.
\end{equation}
By Ger\v{s}gorin's disc Theorem \cite[Thm.~6.1.1]{Horn2013}, every eigenvalue of $\mathbf{R}$ has distance from $1$ that cannot exceed $4/9$. In particular, the minimum eigenvalue of $\mathbf{R}$ is at least $1 - 4/9 = 5/9$, which gives
$y^\top\mathbf{R}y \ge (5/9) \norm{y}^{2}$ for every $y \in \bbR^{\numofpairs}$.
Consider the matrix norm $\norm{\mathbf{E}}\coloneq \sup_{\{x \in \bbR^{\numofpairs}: \norm{x} = 1\}} \norm{ \mathbf{E} x}$
induced by the Euclidean norm on $\bbR^{\numofpairs}$. Then $\norm{\mathbf{E}}$ can be bounded above by the Frobenius norm $\norm{\mathbf{E}}_{\rm F} \coloneq \Big( \sum_{i,j} E_{i,j}^2\Big)^{1/2}$, hence
\begin{equation}
\norm{\mathbf{E}}^2\le \norm{\mathbf{E}}_{\rm F}^2 =  \epsilon_{\numofpairs} \sum_{i = 0}^{\numofpairs-1} (i+1) \cdot 100^{-i}
\le  \epsilon_{\numofpairs} \cdot \sum_{i = 0}^{\infty} (i+1) \cdot 100^{-i}
= \epsilon_{\numofpairs} \cdot \frac{1}{(1 - 1/100)^2} =  \frac{100^2}{99^2} \cdot \epsilon_{\numofpairs} \;.
\end{equation}
Now, for $x, y \in \bbR^{\numofpairs}$, not both zero, we have
\begin{align*}    
(x^{\top},y^{\top})\mathbf{G}\begin{pmatrix}
x\\y
\end{pmatrix}    
&=
x^{\top} \mathbf{D}x + 2 x^{\top}\mathbf{E}y + y^{\top} \mathbf{R} y \\
&\ge
(1 - 3\epsilon_{\numofpairs})\norm{x}^2  - 2 \norm{x}\norm{y}\norm{\mathbf{E}} + \frac{5}{9}  \norm{y}^2\\
&\ge
(1 - 3\epsilon_{\numofpairs} - \norm{\mathbf{E}})\norm{x}^2  + \bigg(\frac{5}{9} - \norm{\mathbf{E}}\bigg) \norm{y}^2\\
&\ge
\bigg(1 - 3\epsilon_{\numofpairs}- \dfrac{100}{99}\sqrt{\epsilon_{\numofpairs}}\bigg)\norm{x}^2  + \bigg(\frac{5}{9} - \dfrac{100}{99}\sqrt{\epsilon_{\numofpairs}}\bigg) \norm{y}^2>0
\end{align*}
since, e.g., $\epsilon_{\numofpairs} \le 1/12 < 1/9$ and so 
\begin{equation}
 1 - 3\epsilon_{\numofpairs}- \dfrac{100}{99}\sqrt{\epsilon_{\numofpairs}}>1 - \frac{1}{3} - \dfrac{100}{3 \cdot 99}= \frac{2\cdot 99 - 100}{3 \cdot 99}>0\;,
\end{equation}
and
\begin{equation}
 \frac{5}{9} - \dfrac{100}{99}\sqrt{\epsilon_{\numofpairs}} > \frac{5}{9} - \dfrac{100}{3 \cdot 99}= \frac{5}{9}\cdot\frac{33 - 20}{33}>0\;.   
\end{equation}
This completes the proof that $\mathbf{G}$ is positive definite.
Now, since $\mathbf{G}$ is a $2\numofpairs \times 2 \numofpairs$ positive definite matrix,
there exists a sequence $(u_i)_{i=1}^{2\numofpairs}$ of linearly independent vectors in $\bbR^{2\numofpairs}$ such that $\mathbf{G}$ is their Gram matrix, i.e., such that $\mathbf{G}=(\ip{u_i}{u_j})_{i,j=1}^{2\numofpairs}$. Indeed, as explained, e.g., in \cite[p.~441]{Horn2013}, it is sufficient to consider the columns of the $2\numofpairs \times 2 \numofpairs$ matrix $\mathbf{G}^{1/2}$.
For every $1 \le i \le 2\numofpairs$, define $v_i \in \bbR^{d}$ by padding $u_i$, i.e., by defining $v_i \coloneq (u_i,0)$ if $d= 2\numofpairs + 1$ or $v_i \coloneq (u_i,0,0)$ if $d= 2\numofpairs + 2$.   Then $(v_i)_{i=1}^{2\numofpairs}$ are linearly independent vectors in $\bbR^{d}$ such that $\mathbf{G}$ is their Gram matrix.
We define
\begin{equation}\label{eq:cov}
a_i\coloneq v_i, \quad b_{i} \coloneq \sqrt{\epsilon_i} \cdot  v_{\numofpairs+i} - v_i \quad \text{for $1 \le i \le \numofpairs$} 
\end{equation}
and what is left is checking that $(a_{i})_{i=1}^{\numofpairs}$ and $(b_{i})_{i=1}^{\numofpairs}$ satisfy the desired relations. The independence is a consequence of \eqref{eq:cov} being an invertible change of variables, with inverse
\[
v_i= a_i, \quad v_{\numofpairs + i} = \frac{1}{\sqrt{\epsilon_i}} \cdot (a_i + b_i)  \quad \text{for $1 \le i \le \numofpairs$} \;.
\]
For every $1 \le i$, $j \le \numofpairs$, using $\mathbf{G}$ to compute the inner products, we get
\begin{align*}
    \ip{a_i}{a_j} &= \ip{v_i}{v_j} = D_{i,j}= \begin{cases} 1 - 3 \epsilon_i &\text{if $i = j$,}\\0 &\text{if $i \neq j$,}\end{cases}\\
    \ip{a_i}{b_j} &= \ip{v_i}{\sqrt{\epsilon_j} \cdot  v_{\numofpairs+j} - v_j } = \sqrt{\epsilon_j} \cdot E_{i,j} - D_{i,j}
    = \begin{cases} 0 -0 \hfill = 0 \hfill&\text{if $i < j$,}\\
    -\epsilon_i - (1 - 3 \epsilon_i) = -(1 - 2 \epsilon_i) &\text{if $i = j$,}\\
    -\epsilon_j - 0 \hfill= -\epsilon_j \hfill&\text{if $i > j$}\;,\end{cases}
    \\ 
    \ip{b_i}{b_j} &= \ip{\sqrt{\epsilon_i} \cdot  v_{\numofpairs+i}}{ b_j} - \ip{v_i}{b_j}\\
    &= \sqrt{\epsilon_i \cdot \epsilon_j} \cdot R_{i,j} - \sqrt{\epsilon_i} \cdot E_{j,i} -\ip{a_i}{b_j}\\
    &= \epsilon_{\min\{i,j\}} \cdot 10^{\abs{i-j}} \cdot R_{i,j} - \sqrt{\epsilon_i} \cdot E_{j,i} -\ip{a_i}{b_j}
    = \begin{cases}
    -2 \epsilon_i                               +\epsilon_i     - 0 \hfill= -\epsilon_i &\text{if $i < j$,}\\
    \epsilon_i                 +\epsilon_i     +(1 - 2 \epsilon_i) = \hfill 1 &\text{if $i = j$,}\end{cases}
\end{align*}
and this completes the proof.
\end{proof}
\subsection{Proof of Claim~\ref{claim:lb-base}}
\begin{proof}[Proof of Claim~\ref{claim:lb-base}]
Note that by \eqref{eq:mi-epsi} we get
\begin{equation}
 1 + \sum_{j = 1}^{\numofpairs} 2 m_j =  1 + 6 \cdot \sum_{j = 1}^{\numofpairs} 100^{\,\numofpairs - j}=  1 + 6 \cdot \frac{ 100^{\,\numofpairs} - 1}{100 -1} = n_0\;.
\end{equation}
For every $0 \le i \le \numofpairs$, define $M_i \coloneq 1 + \sum_{j = 1}^i 2m_j$, which in particular gives $M_0=1$ and $M_{\numofpairs}= n_0$. 
The definition of $g_k$ given in \eqref{eq:gk-in-claim} for every $1 \le k < n_0$ can be rewritten as follows:
\begin{equation}
    g_k = 
    \begin{cases}
    a_i &  \text{for $k$ odd, } \;\;\, M_{i-1} \le k < M_{i}\;,  \\
    b_i &  \text{for $k$ even, }\; M_{i-1} < k < M_{i}\;.
    \end{cases}
\end{equation}
It will be convenient to show that, for every choice of $1 \le j \le \numofpairs$ and every $M_{j - 1} < k \le M_{j}$,
\begin{equation}\label{eq:max}
\max_{u \in \{a_\ell, b_\ell \colon 1 \le \ell \le \numofpairs\}} \ip{u}{x_k-x_{M_{j-1}}} > 0
\end{equation}
and
\begin{subequations}\label{eq:argmax}
        \begin{empheq}[left={\myargmax{u \in \{a_\ell, b_\ell \colon 1 \le \ell \le \numofpairs\}} \ip{u}{x_k-x_{M_{j-1}}}= \empheqlbrace}]{alignat = 2}
    &\{a_j\} \cup \{a_i, b_i \colon j< i \le \numofpairs\}\quad & &\text{for $k$ odd}, \label{subeq:argmax-k-odd}  \\
    &\{b_j\} &  &\text{for $k$ even}. \label{subeq:argmax-k-even}
    \end{empheq}
\end{subequations}
To this end, fix $1 \le j \le \numofpairs$, define $w_j\coloneq a_j + b_j$, and note by Lemma~\ref{lem:ai-bi-def} that:
\begin{align}
\ip{a_i}{w_j} &= \ip{a_i}{a_j} + \ip{a_i}{b_j} =
\begin{cases}
\hfill 0 + 0 \hfill = \phantom{- \epsilon} 0                                            &\text{for}\quad  1 \le i < j\;,\\                  
 1 - 3 \epsilon_j - (1 - 2 \epsilon_j)  = -\epsilon_j & \text{for}\quad  i = j\;,\\                   
\hfill 0 -\epsilon_j \hfill = -\epsilon_j                          &\text{for}\quad  j < i \le \numofpairs \;,      
\end{cases}\\
\ip{b_i}{w_j} &= \ip{b_i}{a_j} + \ip{b_i}{b_j} =
\begin{cases}
\hfill -\epsilon_i -\epsilon_i \hfill= -2\epsilon_i    &\text{for}\quad  1 \le i < j\;,\\            
- (1 - 2 \epsilon_j) + 1 =  \phantom{-}2\epsilon_j   &\text{for}\quad  i = j\;,\\                 
\hfill 0 -\epsilon_j \hfill = \phantom{2}{-\epsilon_j}  &\text{for} \quad j < i \le \numofpairs \;,
\end{cases}
\end{align}

Note that
\[
x_k-x_{M_{j-1}}=  - \sum_{t=M_{j-1}}^{k-1} g_t= -(\underbrace{a_j+b_j + a_j+ b_j +\dots  }_{\text{$k - M_{j-1}$ terms}}),
\]
and counting how many times the term $w_j = a_j+b_j$ appears in the previous sum, we get for every $1 \le i \le \numofpairs$
\begin{align*}
\ip{a_i}{x_k-x_{M_{j-1}}}  &=
\begin{cases}
    -\frac{1}{2}({ k - M_{j-1}}) \ip{a_i}{w_j} & \text{for $k$ odd, }\\
    -\frac{1}{2}({k - 1 - M_{j-1}}) \ip{a_i}{w_j} - \ip{a_i}{a_j} & \text{for $k$ even, }
\end{cases}\\
\ip{b_i}{x_k-x_{M_{j-1}}}  &=
\begin{cases}
    -\frac{1}{2}({ k - M_{j-1}}) \ip{b_i}{w_j} & \text{for $k$ odd, }\\
    -\frac{1}{2}({k - 1 - M_{j-1}}) \ip{b_i}{w_j} - \ip{b_i}{a_j} & \text{for $k$ even, }
\end{cases}\\
\end{align*}
In the case $k$ odd, we obtain
\begin{align}
\ip{a_i}{x_k-x_{M_{j-1}}}  &= -\frac{ k - M_{j-1}}{2} \ip{a_i}{w_j} =
\begin{cases}
0                                       &\text{for}\quad  1 \le i < j\;,\\            
\frac{1}{2}({ k - M_{j-1}}) \epsilon_j    &\text{for}\quad  j \le i \le \numofpairs \;,
\end{cases}\label{eq:ai-dot-xk-k-odd}\\
\ip{b_i}{x_k-x_{M_{j-1}}}  &= - \frac{ k - M_{j-1}}{2} \ip{b_i}{w_j} =
\begin{cases}
(k - M_{j-1})\epsilon_i          &\text{for}\quad  1 \le i < j\;,\\            
-(k - M_{j-1}) \epsilon_j      & \text{for}\quad  i = j\;,\\                 
\frac{1}{2}({ k - M_{j-1}}) \epsilon_j        &\text{for}\quad  j < i \le \numofpairs \;,
\end{cases}\notag
\end{align}
and to find what value is the largest, we exclude the nonpositive values, and we are left to compare the positive quantities $\frac{1}{2}({ k - M_{j-1}}) \epsilon_j $ and  $(k - M_{j-1})\epsilon_i$ for every $1 \le i < j$. We deduce that the former is always larger, since if $i<j$ then  we also have $\epsilon_i< (100 \epsilon_i )/2\le \epsilon_j/2$ by definition, and so \eqref{subeq:argmax-k-odd} holds. Moreover, we get \eqref{eq:max} for $k$ odd.
In the case $k$ even, observe by Lemma~\ref{lem:ai-bi-def} that
\begin{align*}
\ip{a_i}{a_j} &= 
\begin{cases}
0   &\text{for}\quad  1 \le i < j\;,\\            
1 - 3 \epsilon_j   & \text{for}\quad  i = j\;,\\                 
0   &\text{for}\quad  j < i \le \numofpairs \;,
\end{cases}
& \ip{b_i}{a_j} &= 
\begin{cases}
-\epsilon_i   &\text{for}\quad  1 \le i < j\;,\\            
- (1 - 2 \epsilon_j)    & \text{for}\quad  i = j\;,\\                 
0    &\text{for}\quad  j < i \le \numofpairs \;,
\end{cases}
\end{align*}
hence
\begin{align*}
\ip{a_i}{x_k-x_{M_{j-1}}}  &= -\frac{ k -1 - M_{j-1}}{2} \ip{a_i}{w_j} - \ip{a_i}{a_j}
=\begin{cases}
 0   &\text{for}\quad  i < j\;,\\                      
\dots <\frac{1}{2}({ k -1 - M_{j-1}}) \epsilon_j  & \text{for}\quad  i = j\;,\\             
\frac{1}{2}({ k -1 - M_{j-1}}) \epsilon_j  &\text{for}\quad  i > j\;,      
\end{cases}\\
\ip{b_i}{x_k-x_{M_{j-1}}}  &= - \frac{ k - 1 - M_{j-1}}{2} \ip{b_i}{w_j} - \ip{b_i}{a_j}=
\begin{cases}
(k - M_{j-1})\epsilon_i  &\text{for}\quad  1 \le i < j\;,\\                    
1-(k + 1 - M_{j-1}) \epsilon_j   & \text{for}\quad  i = j\;,\\                         
\frac{1}{2}({ k -1 - M_{j-1}}) \epsilon_j  &\text{for}\quad  j < i \le \numofpairs \;,        
\end{cases}
\end{align*}
Note that $M_{j}$ is odd, while $k$ is even, thus $k < M_{j}$. The expression for $\ip{b_i}{x_k-x_{M_{j-1}}}$ is decreasing in $k$ for $i = j$, but increasing in $k$ for $i\neq j$. Recalling that $\epsilon_i\le \epsilon_j/100$ for $i<j$, we get
\[
\ip{b_i}{x_k-x_{M_{j-1}}}=
\begin{cases}
(k - M_{j-1})\epsilon_i  \le \frac{1}{100}(k - M_{j-1})\epsilon_j <  \frac{1}{100}(M_j - M_{j-1}) \epsilon_j  &\text{for}\quad  1 \le i < j\;,\\ 
1-(k + 1 - M_{j-1}) \epsilon_j \ge  1-(M_j  - M_{j-1}) \epsilon_j   & \text{for}\quad  i = j\;,\\
\frac{1}{2}({ k -1 - M_{j-1}}) \epsilon_j < \frac{1}{2}({ M_j - M_{j-1}}) \epsilon_j    &\text{for}\quad  j < i \le \numofpairs \;,
\end{cases}
\]
But by definition we have $M_j - M_{j-1} = 2m_j$ and $\epsilon_j m_j= \frac{1}{4}$, hence $(M_j - M_{j-1})\epsilon_j= 1/2$, and
\[
\ip{b_i}{x_k-x_{M_{j-1}}}=
\begin{cases}
\dots <  \frac{1}{100}\cdot \frac{1}{2} &\text{for}\quad  1 \le i < j\;,\\        
\dots \ge  1-\frac{1}{2} = \frac{1}{2}   & \text{for}\quad  i = j\;,\\             
\dots < \frac{1}{2} \cdot \frac{1}{2}     &\text{for}\quad  j < i \le \numofpairs \;,
\end{cases}
\]
and this gives \eqref{subeq:argmax-k-even} and \eqref{eq:max}.\\
\noindent Our second goal is showing that for every choice of $1 \le j \le \numofpairs$ and every $k$ such that $M_{j - 1} < k \le M_{j}$ 
\begin{equation}\label{eq:max-II}
\max_{u \in \{a_\ell, b_\ell \colon 1 \le \ell \le \numofpairs\}} \ip{u}{x_k} > 0
\end{equation}
and
\begin{subequations}\label{eq:argmax-II}
        \begin{empheq}[left={\myargmax{u \in \{a_\ell, b_\ell \colon 1 \le \ell \le \numofpairs\}} \ip{u}{x_k} = \empheqlbrace}]{alignat = 2}
    &\{a_j\} \cup \{a_i, b_i \colon j< i \le \numofpairs\}\quad & &\text{for $k$ odd}; \label{subeq:argmax-II-k-odd}  \\
    &\{b_j\} &  &\text{for $k$ even}. \label{subeq:argmax-II-k-even}
    \end{empheq}
\end{subequations}
We prove this by finite induction on $j$. In the case $j=1$ note that $x_{M_0}=x_1=0$ and so \eqref{eq:max-II}, resp., \eqref{eq:argmax-II}, is a direct consequence of \eqref{eq:max}, resp., \eqref{eq:argmax}.
Assume the thesis holds for $j - 1$, with $1 \le j - 1 < \numofpairs$, and let $k$ satisfy $M_{j - 1} < k \le M_{j}$. For $u \in \{a_i, b_i \colon 1 \le i \le \numofpairs\}$, we can write
\begin{equation}
   \max_u \ip{u}{x_k} \le \max_u \ip{u}{x_k - x_{M_{j-1}}} + \max_{u}\ip{u}{x_{M_{j-1}}} \;.
\end{equation}
Note that $M_{j-1}$ is odd, and so by inductive hypothesis, the term $\ip{u}{x_{M_{j-1}}}$ is maximized for $u$ in the set $\{a_{j-1}\} \cup \{a_i, b_i \colon j - 1< i \le \numofpairs\}$, which contains every maximizer $u$ of $\ip{u}{x_k - x_{M_{j-1}}}$, as a comparison with \eqref{eq:argmax} shows. It follows that $u$ is a maximizer of $\ip{u}{x_k}$ if and only if it is a maximizer for $\ip{u}{x_k - x_{M_{j-1}}}$, thus
$\max_u \ip{u}{x_k} =  \max_u \ip{u}{x_k - x_{M_{j-1}}} + \max_u \ip{u}{x_{M_{j-1}}}$, which by the induction hypothesis and \eqref{eq:max} gives \eqref{eq:max-II}, and furthermore, by \eqref{eq:argmax} we get \eqref{eq:argmax-II}.
We can now show that for every $1\le k \le n_0$
\begin{equation}\label{eq:f-alternates}
    \hat{f}(x_{k}) =
    \begin{cases}
    \ip{a_i}{x_k} &  \text{for $k$ odd, } \; M_{i-1} \le k \le M_{i},  \\
    \ip{b_i}{x_k} &  \text{for $k$ even, }\; M_{i-1} < k < M_{i}, 
    \end{cases}
\end{equation}
Indeed, for every $1\le i \le \numofpairs$, the cases $M_{i-1} < k \le M_{i}$ are direct consequences of \eqref{eq:max-II} and \eqref{eq:argmax-II},
while the case $k=M_{i-1}$, requires a distinction: if $i = 1$, then $M_0=1$, and using that $x_1=0$ we have
$\ip{a_1}{x_{1}} = 0 = \hat{f}(x_1)$; if $i>1$, then we see by \eqref{subeq:argmax-II-k-odd} that $\ip{a_i}{x_{M_{i-1}}}=\ip{a_{i-1}}{x_{M_{i-1}}}=\hat{f}(x_{M_{i-1}})$.
Recalling the definition of $\hat{f}$, by the property of the subgradients of a maximum of convex functions we see that \eqref{eq:f-alternates} yields $g_k\in \partial \hat{f}(x_k)$ for every $1\le k < n_0$. Noting that, by the linear independence of $\{a_\ell, b_\ell \colon 1 \le \ell \le \numofpairs\}$, all the points $x_k$ for $1\le k < n_0$ are distinct elements of $\bbR^d$, it follows that there exists a well-defined first-order oracle $\Gamma$ for $\hat{f}$ such that $g_k = \Gamma(x_k)$ for every $1\le k < n_0$. Consequently, $x_k = x_k^{\rm sGM}(1,x_1;\Gamma,\bbR^d)$ for every $1\le k \le n_0$, thereby proving that $(x_k)_{k=1}^{n_0}$ is the claimed sequence of sGM iterates.
Finally, we see by \eqref{subeq:argmax-II-k-odd} that $\hat{f}(x_{k}) = \ip{a_{\numofpairs}}{x_{k}}$ if $k$ is odd,
and by \eqref{eq:ai-dot-xk-k-odd} we get
\begin{align*}
\hat{f}(x_{n_0}) &= \ip{a_{\numofpairs}}{x_{n_0}}
=\sum_{j=1}^{\numofpairs} \ip{a_{\numofpairs}}{x_{M_j} - x_{M_{j-1}}} 
=\sum_{j=1}^{\numofpairs} \frac{1}{2} (M_j - M_{j-1})\epsilon_j 
=\sum_{j=1}^{\numofpairs} m_j \epsilon_j
=\sum_{j=1}^{\numofpairs} \frac{1}{4}
=\frac{\numofpairs}{4}\;.
\end{align*}

\end{proof}

\subsection{Proof of Claim~\ref{claim:lb-small-eta}}
\begin{proof}[Proof of Claim~\ref{claim:lb-small-eta}]
Define $a\coloneq \min\{L, R \cdot (2n\eta)^{-1}\}$ and consider the objective $f_a(x) \coloneq a\norm{x}$, which is $a$-Lipschitz (hence $L$-Lipschitz by $a\le L$),  nonnegative, and whose unique minimizer is $x = 0$. Define $x_1\coloneq R v$ for some unit vector $v\in \bbR^d$, so that $\dist(x_1, \cX_\star) = \norm{ x_1 - 0}= R$. Apply sGM initialized at $x_1$ and with stepsize $\eta$ to minimize $f = f_a$ over $\cX= \bbR^d$ and with oracle $g \in \cO_{f_a}$.
The distance between consecutive iterates can be bounded by $a\eta$, hence 
\[
\norm{x_n - x_1}\le (n-1) a \eta <  n a \eta \markrel{rel:lb-small}{\le} \dfrac{R}{2}
\]
where \relref{rel:lb-small} holds by definition of $a$. By the reverse triangle inequality, $x_n$ has distance at least $R/2$ from $x=0$. It follows that
\[
f_a(x_n) - f_a(0)= a\norm{x_n}\ge a \cdot \dfrac{R}{2} = \min\bigg\{ L, \frac{R}{2n\eta}\bigg\} \cdot \dfrac{R}{2} = \min\bigg\{\dfrac{RL}{2}, \frac{R^2}{4n\eta}\bigg\}\;.
\]
\end{proof}
\section*{Acknowledgements}
Tommaso Cesari gratefully acknowledges the support of the Natural Sciences and Engineering Research Council of Canada (NSERC) through grant RGPIN-2023-03688 (Discovery Grants Program).
\section*{AI methodology}
We used OpenAI's GPT-5.5, GPT-5.6 Sol, and GPT-6 Astra to brainstorm and refine some analytical ideas, test preliminary proof arguments, and assist with grammar, punctuation, and phrasing throughout the manuscript. All AI-assisted material incorporated into the manuscript was checked and revised by the authors, who take full responsibility for the paper.
\bibliographystyle{alpha}
\bibliography{biblio_ret}
\end{document}